\documentclass[twoside,11pt]{amsart}

\usepackage[cmtip,all]{xy}
\usepackage{amssymb,latexsym, xspace, enumerate, amsmath, latexsym, times, color, mathtools}
\usepackage{amssymb,amscd}
\usepackage{xcolor}

\usepackage{graphicx}

\def\sqr#1#2{{\vcenter{\hrule height.#2pt
        \hbox{\vrule width.#2pt height#1pt \kern#1pt
                \vrule width.#2pt}
        \hrule height.#2pt}}}

\makeatletter
\newcommand{\mypm}{\mathbin{\mathpalette\@mypm\relax}}
\newcommand{\@mypm}[2]{\ooalign{%
  \raisebox{.1\height}{$#1+$}\cr
  \smash{\raisebox{-.6\height}{$#1-$}}\cr}}
\makeatother

\def\ZZ{{\mathbb Z}}

\def\q{{\mathfrak q}}
\def\p{{\mathfrak p}}

\def\matA{{\mathsf A}}
\def\matD{{\mathsf D}}
\def\matE{{\mathsf E}}
\def\matM{{\mathsf M}}

\def\matA{{\mathsf A}}

\def\matM{{\mathsf M}}

\def\comF{{\mathbf F}}
\def\comD{{\mathbf D}}

\def\D{{\mathsf D}}
\def\E{{\mathsf E}}

\newtheorem{Theorem}{Theorem}[section]
\newtheorem*{MainTheoremA}{Theorem A}
\newtheorem*{MainTheoremB}{Theorem B}
\newtheorem*{MainTheoremC}{Theorem C}
\newtheorem*{MainTheoremD}{Theorem D}
\newtheorem*{MainTheoremE}{Theorem E}
\newtheorem*{MainTheoremF}{Theorem F}

\newtheorem{Lemma}[Theorem]{Lemma}
\newtheorem{Corollary}[Theorem]{Corollary}
\newtheorem{Proposition}[Theorem]{Proposition}

\newtheorem{Notation}[Theorem]{Notation}
\newtheorem{Assumptions and Discussion}[Theorem]{Assumptions and Discussion}

\newtheorem{Remark}[Theorem]{Remark}
\newtheorem{Example}[Theorem]{Example}
\newtheorem{Definition}[Theorem]{Definition}

\newtheorem{Construction}[Theorem]{Construction}

\newcommand{\ul}[1]{\underline{#1}}

\newcommand{\R}{\mathcal{RL}}

\def\Lra{\Longrightarrow}

\newcommand{\be}{\begin{equation*}}
\newcommand{\bee}{\begin{equation}}
\newcommand{\ee}{\end{equation*}}
\newcommand{\eee}{\end{equation}}

\newcommand{\Arg}{\rule{1ex}{1pt}}

\newcommand{\reg}{\operatorname {reg}}
\newcommand{\codim}{\operatorname {codim}}
\newcommand{\Sing}{\operatorname {Sing}}
\newcommand{\Frac}{\operatorname {Frac}}
\newcommand{\Proj}{\operatorname {Proj}}
\newcommand{\Min}{\operatorname {Min}}
\newcommand{\Ass}{\operatorname {Ass}}
\newcommand{\ass}{\operatorname {Ass}}
\newcommand{\ext}{\operatorname {Ext}}

\newcommand{\Jac}{\operatorname {\mathsf{Jac}}}
\newcommand{\depth}{\operatorname {depth}}

\newcommand{\pd}{\operatorname {pd}}
\newcommand{\h}{\operatorname {ht}}
\newcommand{\im}{\operatorname {Im}}
\renewcommand{\ker}{\operatorname {Ker}}
\newcommand{\Spec}{\operatorname {Spec}}

\newcommand{\rk}{\operatorname {rank}}

\newcommand{\straightening}{standardization}
\newcommand{\straighteningsp}{standardizations }
\newcommand{\straightensp}{standardization }

\newcommand{\RLP}{\mathcal{RLP}}
\newcommand{\homog}{\textrm{std}}
\newcommand{\shomog}{\textrm{std}}

\newcommand{\In}{\textrm{in}}

\definecolor{lighterorange}{cmyk}{0,0.42,0.66,0.0}

\colorlet{darkgreen}{green!50!black}
\colorlet{darkblue}{blue!50!black}

\begin{document}


\title{Singularities of Rees-like algebras} 
\author[P. Mantero]{Paolo Mantero}
\address{University of Arkansas, Department of Mathematical Sciences, Fayetteville,
AR 72701}
\email{pmantero@uark.edu}
\author[J. McCullough]{Jason McCullough}
\address{Iowa State University, Department of Mathematics, Ames, IA 50011}
\email{jmccullo@iastate.edu}
\author[L. E. Miller]{Lance Edward Miller}
\address{University of Arkansas, Department of Mathematical Sciences, Fayetteville,
AR 72701}
\email{lem016@uark.edu}

\begin{abstract}
Recently, Peeva and the second author constructed irreducible projective varieties with regularity much larger than their degree, yielding counterexamples to the Eisenbud-Goto Conjecture.  Their construction involved two new ideas: Rees-like algebras and step-by-step homogenization. Yet, all of these varieties are singular and the nature of the geometry of these projective varieties was left open. The purpose of this paper is to study the singularities inherent in this process.  We compute the codimension of the singular locus of an arbitrary Rees-like algebra over a polynomial ring. We then show that the relative size of the singular locus can increase under step-by-step homogenization. To address this defect, we construct a new process, we call prime standardization, which plays a similar role as step-by-step homogenization but also preserves the codimension of the singular locus. This is derived from ideas of Ananyan and Hochster and we use this to study the regularity of certain smooth hyperplane sections of Rees-like algebras, showing that they all satisfy the Eisenbud-Goto Conjecture, as expected. On a more qualitative note, while Rees-like algebras are almost never Cohen-Macaulay and  never normal, we characterize when they are seminormal, weakly normal, and, in positive characteristic, F-split.  Finally, we construct a finite free resolution of the canonical module of a Rees-like Algebra over the presenting polynomial ring showing that it is always Cohen-Macaulay and has a surprising self-dual structure.
\end{abstract}
\maketitle

\section{Introduction}

Given a nondegenerate, embedded projective variety $X$ over an algebraically closed field $k$ corresponding to a homogeneous prime ideal $P \subseteq k[x_0,\ldots,x_n]$, the Eisenbud-Goto conjecture predicts an estimate on the Castelnuovo-Mumford regularity of $X$:
\begin{equation}\label{eq:CMreg}\reg X \leq \deg X - \codim X + 1,\end{equation}
\[ \text{or equivalently}\]
\[ \reg(S/P) \le \deg(S/P) - \mathrm{ht}(P).\]  Equation~\eqref{eq:CMreg} fails for arbitrary schemes, that is, when $P$ is not prime. A surprising  construction introduced by the second author and Peeva \cite{MP} produced the first examples of projective varieties failing this bound by producing from a known embedded scheme with large regularity, a new projective variety embedded in a much larger space which also has large regularity. This reinforces the need to control the singularities of $X$ to ensure optimal estimates for its regularity; in particular, the Eisenbud-Goto conjecture remains open for arbitrary smooth projective varieties or even some mildly singular varieties.  There are many cases where the conjecture  does hold including the case of curves \cite{GLP} and smooth surfaces in characteristic $0$ \cite{P, Laz}, and certain $3$-folds in characteristic $0$ \cite{Ran}. See also related work of Kwak-Park \cite{KP} and Noma \cite{Noma}. There are also mild classes of singular surfaces for which Equation~\eqref{eq:CMreg} holds, see \cite{Niu}. %

\

The process in \cite{MP} of constructing the examples of projective varieties failing Equation~(\ref{eq:CMreg}) involves two major steps. The first step is the construction of the Rees-like algebra, which defines a subvariety of a weighted projective space.  Specifically,  given a homogeneous ideal $I$ in a polynomial ring $S$ over a field $k$, the Rees-like algebra of $I$ is the non-standard graded $k$-algebra $\R(I) \coloneqq S[It, t^2] \subseteq S[t]$.

The second step, which applies to any homogeneous ideal in a non-standard graded polynomial ring, produces an associated ideal in a much larger polynomial ring called its step-by-step homogenization.  Unlike the usual homogenization of an ideal which defines the projective closure of an affine variety, the step-by-step homogenization produces a much larger variety; however, it preserves graded Betti numbers and primeness for nondegenerate primes, making it sufficient to produce the counterexamples to Equation~\eqref{eq:CMreg}.  
\

Thus far, explicit understanding of the geometry of the processes involved in both of these two steps is lacking. It was proved in \cite{MP} that Rees-like algebras are not Cohen-Macaulay but further structure of their singularities is not known. Moreover, the step-by-step homogenization used in \cite{MP} can increase the relative size of the singular locus. The goal of this paper is to better understand the behavior of the singularities and the size of the singular locus after taking each of these two steps. First, we compute the Jacobian of the Rees-like algebra explicitly leading to a complete description of the reduced subscheme structure of the singular locus.

\begin{MainTheoremA}$($Theorem~\ref{thm:MinPrimeJacobian}$)$
Suppose $I$ is a homogeneous ideal in a polynomial ring over a perfect field $k$ with ${\rm char}(k)=p\gg 0$ or ${\rm char}(k)=0$. Set $X = \Proj \R(I)$. Then there is a bijection between the irreducible components of the singular locus and those of the scheme defined by $I$. Moreover,  its codimension in $X$ is 
$$
\codim_X (\Sing X)=  \h(I).
$$
\end{MainTheoremA}

\

Unfortunately, step-by-step homogenization does not preserve the relative size of the singular locus, see Example~\ref{xmp:CodimDropsStepbyStep}.  We introduce a new notion called  {\it prime \straightening}, based on the idea of prime sequences introduced by Ananyan and Hochster in \cite{AH}.  We show that the codimension of the singular locus of an arbitrary variety is preserved after applying  a certain prime standardization.

\begin{MainTheoremB} $($Corollary~\ref{RLhomog}$)$ Suppose $I$ is a homogeneous ideal in a polynomial ring over a perfect field $k$ with ${\rm char}(k)=p\gg 0$ or ${\rm char}(k)=0$.  Set $X = \Proj \R(I)$. There is a prime standardization of the defining prime ideal of the Rees-like algebra $\R(I)$ such that the irreducible components of the singular locus of the resulting variety $Y$ and those of the scheme defined by $I$ are in bijection. Moreover, its codimension in $Y$ is 
$$
\codim_Y (\Sing Y)=  \h(I).
$$
\end{MainTheoremB}

In Section~\ref{canonical}, we also give a complete description of the canonical module. In particular, we give an explicit presentation of $\omega_{\R(I)}$ via linkage theory by fully describing the minimal free resolution of $\omega_{\R(I)}$, including explicit differential maps.  We show that the canonical module has a surprising self-dual resolution. Moreover, we show that, even though the Rees-like algebra is not Cohen-Macaulay when $I$ is not principal, its canonical module is Cohen-Macaulay.

\begin{MainTheoremC} $($Theorem~\ref{canonicalModule}$)$
Suppose $k$ is a field with $\mathrm{char}(k) \neq 2$, and $S$ is a polynomial algebra over $k$.  Let $I$ an ideal with $\h(I) \ge 2$.  Then we can provide an explicit presentation matrix for the canonical module $\omega_{\R(I)}$ of the Rees-like algebra $\R(I)$, and an explicit description of its minimal free resolution. 

In particular, $\R(I)$ is a small Cohen-Macaulay module.
\end{MainTheoremC}

The interested reader may want to consult the statement and proof of Theorem~\ref{canonicalModule} for the precise statement.

\

While the varieties produced in \cite{MP} are highly singular, it is natural to consider the possibility of smooth hyperplane sections of those varieties.  Using the above results and working over characteristic $0$ fields, we exploit Bertini style arguments to show that the resulting smooth varieties satisfy Equation~\eqref{eq:CMreg}.  More precisely, we prove the following:

\begin{MainTheoremD} $($Theorem~\ref{smooth}$)$
If $X$ is an embedded projective scheme defined by a saturated ideal $I$, then there is a regular sequence of general hyperplane sections of a prime \straightensp of the Rees-like algebra of $I$ which is smooth if and only if $X$ is arithmetically Cohen-Macaulay.  Moreover, all such varieties satisfy Equation~\eqref{eq:CMreg}.
\end{MainTheoremD} 

In Section~\ref{sec:Seminormal}, we turn out attention to a more qualitative study of the singularities. Namely, we address weak normality and seminormality of Rees-like algebras. In contrast to the case of Rees algebras, the characterization is surprisingly simple. 

\begin{MainTheoremE} $($Corollary~\ref{cor:criteriaseminormal}$)$
Suppose $k$ is a field with $\mathrm{char}(k) \neq 2$ and $S$ is a polynomial ring over $k$. For an homogeneous $S$-ideal $I$, $I$ is radical if and only if its Rees-like algebra $\R(I)$ is seminormal if and only if $\R(I)$ is weakly normal. 
\end{MainTheoremE}

The rich source of weakly normal Rees-like algebras indicates that the Rees-like construction should be well-behaved with respect to Frobenius splittings.  We prove the following characterization of $F$-split Rees-algebras.

\begin{MainTheoremF} $($Theorem~\ref{thm:Fsplit}$)$
Suppose $k$ is a field with $\mathrm{char}(k) \neq 2$ and $I$ is a radical ideal in a polynomial ring $S$ over $k$.  The ring $S/I$ is $F$-split if and only if $\R(I)$ is $F$-split. 
\end{MainTheoremF}

\

\noindent {\it Acknowledgments:} The authors would like to thank Mark Johnson for many valuable discussions. 

\section{Singular Locus of the Rees-like algebra}

We start by establishing our conventions used through the paper. Unless otherwise stated,  $k$ is a field and $S = k[x_1,\ldots,x_n]$ is a standard graded polynomial ring. We reserve the type face $\matA, \matM, \ldots$ for matrices.   For a specific matrix $\matM$, the notation $I_t(\matM)$ denotes the ideal of $t \times t$-minors. We reserve bold letters $\comF_{\bullet}, \comD_{\bullet}, \ldots$ for chain complexes of modules with differentials $d_{\bullet}^{\comF}, d_{\bullet}^{\comD}, \ldots$
Whenever there is a specified system of generators $g_1,\ldots,g_t$ for an ideal $H$, we simply write $\Jac(J)$ for the Jacobian matrix $\Jac(g_1,\ldots,g_t)$ (e.g. in Theorem~\ref{thm:JacobianRL}).

\

Fix a homogeneous $S$-ideal $I$ with choice of generators $I = (f_1,\ldots,f_m)$. The Rees-like algebra of $I$ is the algebra $S[It, t^2] \subseteq S[t]$, where $t$ is a new variable. We  denote the Rees-like algebra by $\R(I)$. It has an explicit presentation as a quotient of a non-standard graded polynomial ring over $S$, namely $\R(I)\cong T/\RLP(I)$ where $T \coloneqq S[y_1,\ldots,y_m, z]$ has grading defined by $\deg y_i = \deg f_i + 1$ and $\deg z = 2$, and $\RLP(I)$ is a homogeneous ideal of $T$. The usefulness of Rees-like algebras lies in the detailed understanding of the kernel, $\RLP(I)$, of the map of $k$-algebras $T \to \R(I)$ given by $y_i \mapsto f_i t$ and $z \mapsto t^2$ as summarized in the following theorem. 

\begin{Theorem}[{McCullough and Peeva \cite[Theorem 1.6, Proposition 2.9]{MP}}]\label{RLbasics}
The ideal $\RLP(I)$ is the sum $\RLP(I)_{\textrm{syz}} + \RLP(I)_{\textrm{gen}}$ with generators\\
\begin{eqnarray*}
\RLP(I)_{\textrm{syz}} &= & \left\{ r_j \mathrel{\mathop:}= \sum_{i=1}^m c_{ij} y_i\,  \mid \, \sum_{i=1}^mc_{ij}f_i = 0 \right\} \quad \mbox{ and }\quad  \\
\RLP(I)_{\textrm{gen}} & =& \{y_iy_j - zf_if_j \,\mid\, 1 \leq i,j \leq m\}.
\end{eqnarray*}
Moreover, 
\begin{enumerate}
\item  The maximal degree of a minimal generator of $P$  is
$$\hbox{\rm maxdeg} (P) = \max\Bigg\{\,1+\hbox{\rm maxdeg} \Big(\mathrm{Syz}_1^S(I)\Big),\ 2\Big(\mathrm{maxdeg}(I)+1\Big)\,\Bigg\}\,.$$ 
 \item The multiplicity or degree of $T/\RLP(I)$   is $$\deg(T/\RLP(I))=2\,   \prod_{i = 1}^m\Big( \deg(f_i)+1\Big)\,.$$ 
  \item The Castelnuovo-Mumford regularity,
  the  projective dimension, the depth, the codimension, and the dimension of $T/\RLP(I)$ are:
  \begin{itemize}
 \item $\reg(T/\RLP(I))= \reg(S/I) +2 + \sum_{i=1}^m\,\deg(f_i)  $
 \item $\pd(T/\RLP(I))=\pd(S/I)+m-1$ 
 \item $\mathrm{depth} (T/\RLP(I))=\mathrm{depth} (S/I)+2$ 
\item $\h (\RLP(I))=m$
\item $\dim (T/\RLP(I))=n+1$.
\end{itemize}
  \end{enumerate}
  \end{Theorem}
\noindent In the previous theorem, ${\rm maxdeg} M$ denotes the maximal degree of an element in a minimal system of generators of $M$.

Our study of the singular locus of a Rees-like algebra $\R(I)$ is based on an explicit description of the Jacobian matrix $\Jac(\RLP(I))$ via computing a block decomposition. Some of the blocks will be essentially $m \times \binom{m+1}{2}$-sub matrices of the Jacobians of the ideals $(f_1,\ldots,f_m)^2$ (resp. $(y_1,\ldots,y_m)^2$) consisting of rows using only partials corresponding to the variables $x_1,\ldots,x_m$ (resp. $y_1,\ldots, y_m$). Specifically, 
\begin{itemize}
\item denote by $\Jac_{\ul{x}}\left((f_1,\ldots,f_m)^2\right)$ the Jacobian matrix of $(f_1,\ldots,f_m)^2$ with respect to $x_1, \ldots, x_m$, and 
 \item denote by $\Jac_{\ul{y}}\left((y_1,\ldots,y_m)^2\right)$ the Jacobian matrix of $(y_1,\ldots,y_m)^2$ with respect to $y_1, \ldots y_m$. 
\end{itemize}

Another block is described by a minimal free resolution $\comF_\bullet$ of $(f_1,\ldots,f_m)$. Specifically, denote by $d_1^{\comF} \coloneqq d_1^{\comF}(\ul{f})=(c_{ij})$ the first differential in $\comF$, i.e., the matrix whose columns are the first syzygies of the $f_i$. Finally, let  $\matA=(a_{kj})$, where $a_{kj}=\partial_{x_k}r_j=\sum_{i=1}^m \partial_{x_k}(c_{ij})y_i$. With this notation, we may describe the Jacobian $\Jac(\RLP(I))$.

\begin{Proposition}\label{thm:JacobianRL} Using the notation above, up to reordering of the columns and rows, the Jacobian matrix of $\RLP(I)$  has a block decomposition
\be
\Jac(\RLP(I))=\left( \begin{array}{c|| c|c}
 & \mbox{ generators in }\RLP(I)_{\textrm{syz}}	& \mbox{ generators in } \RLP(I)_{\textrm{gen}} \\  \hline\hline
 \vdots &    & \\
\partial_{x_i} & \matA &  - z {\Jac}_{\ul{x}}((f_1,\ldots,f_m)^2) \\
 \vdots &   & \\\hline
 \vdots &  & \\
\partial_{y_j} & d_1^{\mathbf{F}} & {\Jac}_{\ul{y}}((y_1,\ldots,y_m)^2) \\
 \vdots &   & \\\hline
\partial_z & 0 & -f_if_j  
\end{array}  
  \right).
\ee

\end{Proposition}

\begin{proof}
We order the rows as follows. The first $n$ rows correspond to partial differentiation with respect to the variables $x_1,\ldots,x_n$, the next $m$ rows correspond to $\partial_{y_i}$ for $i=1,\ldots,m$, and the last row corresponds to $\partial_z$. The first $b=\rk(\comF_0)$ columns correspond to the minimal generators $r_1,\ldots,r_b$  in the set $\RLP(I)_{\textrm{syz}}$ described in Theorem~\ref{RLbasics}. The following $\binom{m+1}{2}$ columns correspond to the generators in the set $\RLP(I)_{\textrm{gen}}$.

For the blocks within the first $b$ columns, writing $r_j=\sum_{i=1}^m c_{ij}y_i$ for some $1\leq j \leq b$, by linearity we have 
\be
\partial_{x_k}r_j=\sum_{i=1}^m \partial_{x_k} (c_{ij} y_i) =\sum_{i=1}^m y_i\partial_{x_k} (c_{ij}),
\ee
and clearly $\partial_{y_k}r_j=\sum_{i=1}^m \partial_{y_k} (c_{ij} y_i)=c_{kj}$ and $\partial_zr_j=\sum_{i=1}^m \partial_{z} (c_{ij} y_i)=0$.

\

For the blocks concerning the last $\binom{m+1}{2}$ columns,  set $b_{ij}:=y_iy_j-zf_if_j$ for $1 \leq i\leq j \leq m$. The following calculations finish the proof:
\begin{eqnarray*}
\partial_{x_k}b_{ij} &=& \partial_{x_k}(y_iy_j)-\partial_{x_k}(zf_if_j)  =    -z\partial_{k}(f_if_j), \\
\partial_{y_k}b_{ij} & =& \partial_{y_k}(y_iy_j) - \partial_{y_k}(zf_if_j) = \partial_{y_k}(y_iy_j), \text{ and } \\ 
\partial_zb_{ij} & = & \partial_{z}(y_iy_j) - \partial_{z}(zf_if_j)=-f_if_j. 
\end{eqnarray*}
\end{proof}

\begin{Example}\label{xmp:Jacht2CI}
Let $I=(x_1,x_2)\subseteq k[x_1,x_2]$ and $\R(I)$ be its Rees-like algebra with defining ideal 
\be
\RLP(I) = (-y_1x_2+x_1y_2,\;y_1^2 - zx_1^2, \; y_1y_2 - zx_1x_2,\;y_2^2 - zx_2^2).
\ee
The Jacobian $\Jac(\RLP(I))$ is the following matrix with $4$ columns and $5$ rows. 
\be
\Jac(\RLP(I)) =  \left(\begin{array}{l || c | ccc}
		    & -y_1x_2+x_1y_2 & y_1^2 - zx_1^2 & y_1y_2 - zx_1x_2 & y_2^2 - zx_2^2\\  \hline\hline
\partial_{x_1} & y_2 & 	-2x_1z	& -x_2z 	& 0		\\
\partial_{x_2} & -y_1 &  	0		& -x_1z	& -2x_2z	\\ \hline
\partial_{y_1} & -x_2 &     	2y_1		& y_2	& 0		\\
\partial_{y_2} & x_1 &  	0		& y_1	& 2y_2	\\  \hline
\partial_z & 0	&  	-x_1^2	& -x_1x_2	& -x_2^2
\end{array}\right)
\ee

\end{Example}

With this explicit description of the Jacobian in Proposition~\ref{thm:JacobianRL}, we determine in Theorem~\ref{thm:MinPrimeJacobian} the codimension of the singular locus of the Rees-like algebra of any ideal $I$ when $2$ is unit. Interestingly, this number only depends on the height of $I$ and the number of generators of $I$. Recall, by Theorem~\ref{RLbasics},  $\h(\RLP(I))=\mu(I)$. 

\begin{Theorem}\label{thm:MinPrimeJacobian}
Let $k$ be a field with ${\rm char}(k)=0$. 
Let $\R(I)$ be the Rees-like algebra of a nonzero, proper ideal $I$ with  minimal primes $\Min(I) = \{\p_1,\ldots,\p_r\}$. Let $\RLP(I)$ be the defining ideal of $\R(I)$, let $X = \Proj \R(I)$, and let $J$ be the radical ideal of $T$ defining $\Sing(X)$, then 

 $$J =\left( \p_1 + (\underline{y})\right) \cap  \left(\p_2 + (\underline{y}) \right) \cap \cdots \cap \left( \p_r + (\underline{y}) \right).$$  In particular
\begin{itemize}
\item there is a one-to-one correspondence between $\Min(I)$ and $\Min(J)$,
\item $\h(J) = \mu(I) + \h(I)$, and  
\item 
$
\codim_X (\Sing X)=  \h(I).
$
\end{itemize}
\end{Theorem}

\begin{proof} \sloppy Set $m = \mu(I)$. The Jacobian criterion states that $\Sing(\R(I))$ is scheme theoretically defined by the ideal 
 $I_m(\Jac(\RLP(I)))$, 
so it suffices to show $$\Min(I_m(\Jac(\RLP(I)))) = \{\p_i+(\ul{y})\,\mid\,i=1,\ldots,r\}.$$

We use the notation of Proposition~\ref{thm:JacobianRL} for the Jacobian matrix $\Jac(\RLP(I))$. First note that there are no containments among the set of primes of the form $\p_i+(\ul{y})$ as clearly there are no containments among the ideals $\p_i$ and these ideals are transversal with the ideal $(\ul{y})$. To prove that $\Min(J) = \{\p_i+(\ul{y})\,\mid\, i=1,\ldots,r\}$ we start by showing that $J \subset \p_i+(\ul{y})$ for all $1 \leq i \leq r$. Fix arbitrary such $i$ and invoke Theorem~\ref{thm:JacobianRL} to observe that 
\begin{itemize}
\item all entries of $\matA$ and $\Jac_{\ul{y}}((y_1,\ldots,y_m)^2)$ lie in $(y_1,\ldots,y_m)$,
\item all entries of the block matrices $-z\Jac_{\ul{x}}((f_1,\ldots,f_m)^2)$ and $(-f_if_j)$ lie in $(f_1,\ldots,f_m)$.
\end{itemize}
Thus, any $m$-minor of $\Jac(\RLP(I))$ involving one of the last $\binom{m+1}{2}$ columns or one of the rows corresponding to $\partial_{x_j}$ or $\partial_z$ is contained in $I+(\ul{y}) \subseteq \p_i+(\ul{y})$. The remaining $m$-minors  generate $I_m(c_{ij})$, which is contained in $\sqrt{I}$  (e.g. \cite[Thm~2.1(b)]{BE74}) and thus in $\p_i\subseteq \p_i+(\ul{y})$ .

For the converse, let $\q$ be a prime ideal containing $J=\sqrt{I_m(\Jac(\RLP(I)))}$. By Proposition~\ref{thm:JacobianRL}, $\q$ contains the ideal of $m$-minors $I_m(\Jac_{\ul{y}}((y_1,\ldots,y_m)^2))$ and so it contains the ideal $(2y_1^m,\,2y_2^m,\ldots,2y_m^m)$. Since $\q$ is prime and ${\rm char}(k)\neq 2$, then $(y_1,\ldots,y_m) \subseteq \q.$ 

To finish the proof we show that $\q$ contains one of the $\p_i$'s. The ideal $(f_1,f_2,\ldots,f_m)^2 \cdot I_{m-1}(c_{ij})$ is the ideal generated by all $m$-minors determined by the last row, $(m-1)$ of the $m$ rows corresponding to $\partial_{y_k}$, $(m-1)$ of the first $b$ columns (corresponding to the generators in $\RLP(I)_{\textrm{syz}}$), and one column among the last $\binom{m+1}{2}$ (corresponding to one of the generators in $\RLP(I)_{\textrm{gen}}$). As such, we have 
\be
(f_1,f_2,\ldots,f_m)^2 \cdot I_{m-1}(c_{ij})\subseteq I_m(\Jac(\RLP(I))).
\ee
Since $(c_{ij})=(d_1^{\comF})$, by  \cite[Thm~2.1(b)]{BE74} we have 
$\sqrt{I} = \sqrt{I_{m-1}(c_{ij})}$, thus 
taking radical of both sides in the above inclusion, and noticing that the radical of the left-hand side is simply $\sqrt{I}$ and the radical of the right-hand side is $J$, we finally obtain
\be
\p_1 \cap \p_2 \cap \cdots \cap \p_r  = \sqrt{I} = \sqrt{(f_1,f_2,\ldots,f_m)^2 \cdot \sqrt{I}} \subseteq J\subseteq \q.
\ee
Therefore, $\q$ contains one of the $\p_i$. This concludes the proof.

\end{proof}

\begin{Remark} 
A very similar argument proves the statement over any perfect field of ${\rm char}(k)=p>2$ as long as $p$ does not divide the degree of any minimal generator of $I$. So for instance the result will still hold true if ${\rm char}(k)=p\gg 0$.
\end{Remark}

When $S$ is a polynomial ring over a field, there are more conceptual proofs of Theorem~\ref{thm:MinPrimeJacobian}. Specifically, K. E. Smith noted in preliminary discussions with us that as $S[It, t^2]$ has a smooth normalization given by $S[t]$, the singularities are relatively mild and defined by the conductor ideal, which can be shown to be $I + It$. However, our explicit approach to the Jacobian also gives similar results for Rees-like algebras of ideals in quotients of polynomial rings. As an example, analogous arguments to those proving Theorem~\ref{thm:MinPrimeJacobian} can be used to prove the following result where the ground ring $S$ is not regular.

\begin{Theorem}\label{isol} 
Let $S={\mathbb C}[x_1,\ldots,x_n]/\p$ where  $\p$ is a non-degenerate homogeneous prime ideal with ${\Sing \Proj}(S) \subseteq \{(x_1,\ldots,x_n)S\}$. Let $I$ be a homogeneous $S$-ideal.

If the presentation matrix of $I$ as an $S$-ideal contains no linear forms, then 
the same conclusions of Theorem \ref{thm:MinPrimeJacobian} holds with the adjustment that also $(x_1,\ldots,x_n)$ is a minimal prime of the singular locus of the Rees-like algebra.
\end{Theorem}

In Theorem \ref{isol}, the assumption on the presentation matrix of $I$ is needed, as the following example illustrate.
\begin{Example}
Assume  ${\rm char}(k)\neq 2$ and let $S=k[x_1,x_2,x_3]/(x_1^2-x_2x_3)$ and $I=(x_1,x_2)S$. Observe that $I$ has the linear syzygies $(x_2, -x_1)$ and $(x_1, -x_3)$. The singular locus of $\R(I)$ has only one minimal prime, which is $(x_1,x_2,y_1,y_2)$.
\end{Example}

\begin{proof}
As in the first part of the statement of Theorem \ref{RLbasics}, one has $\R(I) \cong k[x_1,x_2,x_3,y_1,y_2,z]/Q$, where 
$$Q = (x_1y_2-x_2y_1, x_1y_1-x_3y_2, y_1^2-x_1^2z, y_1y_2-x_1x_2z, y_2^2-x_2^2z, x_1^2-x_2x_3).$$
Then, the Jacobian matrix of $Q$ over $k[x_1,x_2, x_3 ,y_1,y_2,z]$ is
\be
\matM=\left(\begin{array}{l || cc | ccc | c}
\partial_{x_1} & y_2 & y_1		&	-2x_1z	& -x_2z 	& 0		& 2x_1 \\
\partial_{x_2} & -y_1 &  0		&	0		& -x_1z	& -2x_2z	& -x_3 \\ 
\partial_{x_3} & 0	 &  -y_2	&	0		& 	0	& 	0	& -x_2 \\ \hline
\partial_{y_1} & -x_2 &     x_1	& 2y_1		& y_2	& 0		& 0 \\
\partial_{y_2} & x_1 & -x_3	& 0			& y_1	& 2y_2	&  0 \\  \hline
\partial_z 	     & 0	& 0		&  -x_1^2	& -x_1x_2	 & -x_2^2 & 0
\end{array}\right)
\ee
and set $J=\sqrt{I_3(\matM)}$. It easy to check that $J\subseteq (x_1,x_2,y_1,y_2)$.

\noindent Conversely, let $\q \in \Min(J)$. Notice that $$\det\left(\begin{array}{ccc}
y_2 & -2x_1z 	& 2x_1 \\
x_1 & 0		& 0	   \\
0     & -x_1^2	& 0
\end{array}\right) = -2x_1^4\in J.$$ Similarly, $2x_2^4\in J$. Since ${\rm char}(k)\neq 2,$ then $x_1,x_2\in \sqrt{J}\subseteq q$. Moreover,
$$
\det\left(\begin{array}{ccc}
-2x_1z & -x_2z 	& 2x_1 \\
2y_1 & y_2	& 0	   \\
0     & y_1		& 0
\end{array}\right)=4x_1y_1^2 \in J.$$ Similarly, $x_2y_1^2, x_3y_1^2, x_2y_1^2,  x_2y_2^2,$ and $x_3y_2^2$ lie in $J$. Thus one has an inclusion $(y_1^2, y_2^2)(x_1,x_2,x_3)\subseteq J$. 

Now assume by contradiction that $(y_1,y_2)\not\subseteq \q$. One then has  $(x_1,x_2,x_3)\subseteq \q$. Reducing the entries of $\matM$ modulo $(x_1,x_2,x_3)$ one sees that $(2y_1^3, 2y_2^3)\subseteq \q$. This shows that $(y_1,y_2)\subseteq \sqrt{J}\subseteq q$, which is a contradiction.

\end{proof}

We close this section with two examples showing that describing all of the associated primes of the singular locus of a Rees-like algebra would be complicated, even in relatively simple examples.

\begin{Example}
Let $I=(x_1,x_2)\subseteq k[x_1,x_2]$ and let $\R(I)$ be its Rees-like algebra. One can easily check 
\be
\h(I_2(\Jac(\RLP(I)))) =2+2=4
\ee
and 
$\Min(I_2(\Jac(\RLP(I)))) = \{(x_1,x_2,y_1,y_2)\}$ however $$\Ass(I_2(\Jac(\RLP(I)))) =  \{(x_1,x_2,y_1,y_2), (x_1,x_2,y_1,y_2,z)\}.$$

\end{Example}

\begin{Example}
Let $X_{2\times 3}$ be a generic 2 by 3 matrix and $I=I_2(X_{2\times 3})\subseteq k[X_{2\times 3}]$; it is well-known that $I$ is prime. Let $\R(I)\cong T/\RLP(I)$ be its Rees-like algebra, then 
$\Ass(I_3(\Jac(\RLP(I)))) = \{I+(\ul{y}),\; I+(\ul{y}, z),\; (x_{ij}, \ul{y})\}$ whereas $\Min(I_3(\Jac(\RLP(I)))) =\{I+(\ul{y})\}.$

\end{Example}

\section{\straightening s}\label{hom}

The usual way one homogenizes a non-homogeneous ideal $I \subset k[x_1,\ldots,x_n]$ is by adjoining a new variable, say $w$, and homogenizing all terms of all elements of the ideal by multiplying by the appropriate power of $w$ to make the element homogeneous.  This corresponds to taking the projective closure of $V(I)$ in $\mathbb{P}^n_k$.  Thus the resulting homogeneous ideal is prime but this process does not preserve the structure of free resolution of the corresponding ideal. An alternate method of constructing standard graded analogues of nonstandard graded prime ideals, called step-by-step homogenization in \cite[Theorem 4.5]{MP}, preserves primeness for nondegenerate prime ideals and graded Betti numbers at the expense of adding many more variables. For each variable $x$ with $\deg(x) = d > 1$, one appends a new variable $u$, sets $\deg(u) = \deg(x) = 1$ and replaces every instance of $x$ with $xu^{d-1}$.  As the role of this process is to transform a non-standard graded ring into a standard graded one, we refer to it as a \straightening.

\begin{Definition}
Suppose $T$ is a positively graded polynomial ring over a field $k$. A {\bf standardization} of $T$ is a graded, flat map $(\Arg)^\homog \colon T \to T^\homog$ of graded $k$-algebras, where $T^\homog$ is standard graded polynomial ring over $k$.  For an ideal $I = (f_1,\ldots,f_m) \subseteq T$, write $I^\homog$ for the $T^\homog$-ideal $(f_1^\homog,\ldots,f_m^\homog)$.
\end{Definition}

Thus step-by-step homogenization is a standardization that has the additional property that for any nondegenerate prime ideal  $Q$ of $T$, the ideal $Q^\homog$ is also prime.  Any standardization will thus increase the number of variables and thereby increase the size of the singular locus of the corresponding varieties. However, it is desirable that the codimension of the singular locus is preserved. Unfortunately, step-by-step homogenization does not preserve it. 

\begin{Example}\label{xmp:CodimDropsStepbyStep}
Let $Q= I_2\left[ \begin{array}{ccc}
x & y & z \\
u & v & w \end{array}\right] \subset S = k[u,v,w,x,y,z]$, with the non-standard grading given by setting $$\deg(x) =\deg(y)= \deg(z) =2 \text{ and } \deg(u)=\deg(v)=\deg(w)=1.$$ Consider the step-by-step \straightensp given by the ring map $$S \to S^\homog := k[u,v,w,x_1,x_2,y_1,y_2,z_1,z_2]$$ given by $x \mapsto x_1 x_2$, $y \mapsto y_1 y_2$, and $z \mapsto z_1 z_2$. The image of $Q$ is $$Q^\homog=(x_1x_2v-y_1y_2u, x_1x_2w-z_1z_2u, y_1y_2w-z_1z_2v).$$ One may easily verify that $\h(Q) = 2$, $\h(Q^\homog) =2$, and $\h(I_2(\Jac(Q))) =  6$ yet $I_2(\Jac(Q^\homog))$ has height $5$. 
One can also build examples of Rees-like algebras whose singular locus codimension fails to be preserved in a similar fashion.
\end{Example}

We adapt work of Ananyan and Hochster to define new \straighteningsp that preserve the relative size of the singular locus.  Following \cite{AH}, we define a sequence of elements $g_1,\ldots, g_t \in S$ to be a {\bf prime sequence} provided $(g_1,\ldots,g_t)$ is a proper ideal and $S/(g_1,\ldots,g_i)$ is a domain for all $1 \le i \le t$.  Clearly any prime sequence is a regular sequence.  The following near converse is implicit in their work.

\begin{Lemma} \label{primeseq}
Let $S$ be a standard graded polynomial ring and let $g_1,\ldots,g_t \in S$ be a homogeneous regular sequence of elements of positive degree.  If $I$ is prime, then $g_1,\ldots,g_t$ is a prime sequence.  Moreover, any permutation of $g_1,\ldots,g_t$ is a prime sequence.
\end{Lemma}

\begin{proof} Proceed by contradiction and set $I_i = (g_1,\ldots,g_i)$. Pick $i$ maximal so that $I_i$ is not prime, so $i < t$.  Pick homogeneous elements $a$ and $b$ in $S \setminus I_i$ with $ab \in I_i$ and with $\deg(ab)$ minimal.  Since $I_{i+1}$ is prime, without loss of generality we may assume $a \in I_{i+1}$.  Writing $a = \sum_{j = 1}^{i+1} s_j g_j$, we have $$bs_{i+1}g_{i+1} = b(a - \sum_{j=1}^i s_j g_j) \in I_i.$$  Since $g_1,\ldots,g_t$ is a regular sequence, $bs_{i+1} \in I_i$.  Also $\deg(bs_{i+1}) < \deg(ab)$.  By the minimality assumption, this gives $s_{i+1} \in I_i$ and hence $a \in I_i$, which is a contradiction.
\end{proof}

The usefulness of this idea is contained in the following result, which is essentially the content of \cite[Cor. 2.9, Prop. 2.10]{AH}.

\begin{Proposition}[Ananyan and Hochster]\label{Q^h} Assume $k$ is algebraically closed and let $S = k[x_1,\ldots,x_n]$.  Suppose $g_1,\ldots,g_t$ is a homogeneous prime sequence in $S$ and set $R = k[g_1,\ldots,g_t]$.  Suppose $I \subset R$ is a homogeneous ideal.
\begin{enumerate}
\item The ideals $I$ and $IS$ have the same graded Betti numbers.
\item For $\p \in \Spec(R)$, $\p \in \ass(R/I)$ if and only if $\p S \in \ass(S/IS)$. 
\item In particular, if $I$ is prime, then $IS$ is prime.
 \item If $I=\q_1\cap \cdots \cap \q_r$ is any primary decomposition of $I$, then $\q_1S\cap \cdots \cap \q_rS$ is a primary decomposition of $IS$.
\end{enumerate} 
\end{Proposition}

Homogeneous prime sequences give rise to \straighteningsp and we make the following definition.

\begin{Definition}
Suppose $T$ is a positively graded polynomial ring. A {\bf prime standardization} of $T$ is a standardization $(\Arg)^\homog \colon T \to T^\homog$ such that for every prime ideal $P \subseteq T$, $P^\homog$ is prime.
\end{Definition}

To see the connection with prime sequences, we note the following:

\begin{Proposition}
Let $T = k[x_1,\ldots,x_n]$ be a positively graded polynomial ring and $(\Arg)^\homog \colon T \to T^\homog$ a standardization.  Let $g_i := x_i^\homog$.  Then $(\Arg)^\homog$ is a prime standardization if and only if $g_1,\ldots,g_n$ is a prime sequence.
\end{Proposition}

\begin{proof} The ``if'' direction follows from Proposition~\ref{Q^h}.  For the ``only if'' direction, suppose $(\Arg)^\homog$ is a prime standardization.  For every $1\leq i \leq t$, by Prop.~\ref{Q^h}(3) we have $(g_1,\ldots,g_i) = (x_1,\ldots,x_i)^\homog$ is a prime ideal thus, by Lemma \ref{primeseq}, $g_1,\ldots,g_n$ is a prime sequence.
\end{proof}

By our definition, step-by-step homogenization is not a prime standardization since nonlinear monomials do not form a prime sequence. We now show that there is always a choice of prime standardization that, unlike step-by-step homogenization, preserves the codimension of the singular locus of any ideal.  First we fix a chosen prime standardization.

\begin{Construction}\label{primeH}
Let $T=k[t_1,\ldots,t_n]$ be a positively graded polynomial ring over an algebraically closed field $k$ with $\deg(t_i)=d_i \in \mathbb{Z}_+$. 
Set $$W=\{w_{i,j,\ell}\,|\,1 \le i \le n, 0 \le j \le n, 1 \le \ell \le d_i \},$$ a set of new variables.  Set $T^{\shomog} \coloneqq k[W]$ and let $F_i = \sum_{j = 0}^n \prod_{\ell = 1}^{d_i} w_{i,j,\ell} \in T^{\shomog}$, where we define $\deg(w_{i,j,\ell}) = 1$ for all $i,j,\ell$.  Define the graded map of rings $(\Arg)^\homog:T \to T^\homog$ by setting $t_i^\homog = F_i$.  Since each $F_i$ is irreducible, say by Eisenstein's criterion, and since the variables appearing in $F_i$ are disjoint from those of $F_j$ for $i \neq j$, it follows from Lemma~\ref{primeseq} that $F_1,\ldots,F_n$ form a prime sequence.
\end{Construction}

\begin{Example}
Let $T = k[x_1,x_2,x_3]$, where $\deg(x_i) = i$ for $i = 1,2,3$.  The ideal $I = (x_1^2 - x_2, x_1^3 - x_3)$ of $T$ is then homogeneous.  The prime \straightensp $I^\shomog$ of $I$ from Construction~\ref{primeH} is then generated by the following two elements\\

\begin{tabular}{lll}
$f$ &$=$ &$(w_{1,0,1} + w_{1,1,1} + w_{1,2,1} + w_{1,3,1})^2$ \\
&&$- \,(w_{2,0,1}w_{2,0,2} + w_{2,1,1}w_{2,1,2} + w_{2,2,1}w_{2,2,2} + w_{2,3,1}w_{2,3,2}),$\\
\\
$g$ &=  &$(w_{1,0,1} + w_{1,1,1} + w_{1,2,1} + w_{1,3,1})^3  - (w_{3,0,1}w_{3,0,2}w_{3,0,3} $\\
&&$+ \,w_{3,1,1}w_{3,1,2}w_{3,1,3} + w_{3,2,1}w_{3,2,2}w_{3,2,3} + w_{3,3,1}w_{3,3,2}w_{3,3,3}).$
\end{tabular}
\end{Example}

\noindent By convention, we set the height of the unit ideal to be $\h((1)) = \infty$.

\begin{Lemma}\label{partials} For the \straightensp defined in Construction~\ref{primeH}, the ideal $$\left(\frac{\partial}{w_{i,j,\ell}}(F_i)\,|\, 0 \le j \le n, 1 \le \ell \le d_i\right)$$ has height at least $n+1$.
\end{Lemma}

\begin{proof} The generators of the form $\frac{\partial}{\partial_{w_{i,j,1}}}(F_i) = \prod_{\ell = 2}^{d_i} w_{i,j,k}$ for $0 \le j \le n$ constitute a regular sequence, as they are expressed in disjoint sets of variables. 
\end{proof}

We adopt the notation $\codim(\Sing(X)):=\codim_X(\Sing(X))$. We say that an ideal is {\em unmixed} if all its associated primes have the same height.

\begin{Theorem}\label{LJP} Let $T$ be a positively graded polynomial ring over an algebraically closed field $k$, let $I$ be any homogeneous equidimensional ideal of $T$. Assume either ${\rm char}(k)=0$ or ${\rm char}(k)=p>0$ does not divide the degree of any minimal generator of $I$. Denote by $(\Arg)^\shomog$ the prime \straightensp in Construction~\ref{primeH}. 

For $X = \Proj (T/I)$ and $X^{\homog} = \Proj (T^{\homog}/I^{\homog})$,
$$\codim_X(\Sing X)=\codim_{X^{\homog}}(\Sing X^{\homog})$$ and there is a bijection between $\Min(\Sing(I))$ and the minimal primes of $\Sing(I^\homog)$ of height at most $\dim(T)$.\\
\end{Theorem}

\begin{proof} 
We first prove the case where ${\rm char}(k)=0$.  Let $T = k[x_1,\ldots,x_{n-1},y]$.  
By induction we may focus on the case where we replace a single variable $y$ of degree $d$ by $F = \sum_{j = 0}^{n} \prod_{\ell = 1}^{d} w_{j,k}$ and leave all other variables fixed.  Let $I = (g_1,\ldots,g_s)$ be a homogeneous ideal of $T$ and let $I^\homog$ denote the ideal generated by the images $G_i = g_i^\homog$ of the $g_i$ under the map $(\Arg)^\homog:T \to T^\homog = k[z_1,\ldots,z_N,w_{0,1},\ldots,w_{n,d}]$ defined by $z_i \longmapsto z_i$ and $y \longmapsto F$.  
Let $c=\h(I)$. By Lemma \ref{Q^h}(3) we know that $c=\h(I^\homog)$ as well. 
By the Jacobian criterion, $\Sing(X)$ and $\Sing(X^{\homog})$ are defined, up to radical, by $\h(I_c(\Jac(I)))$ and $\h(I_c(\Jac(I^\homog)))$, respectively. Write

\be
\Jac(g_1,\ldots,g_s)= \left(\begin{array}{l || cccc}
		    & g_1 & g_2 & \ldots & g_s \\  \hline\hline
\partial_{z_1} & \partial_{z_1}(g_1) & 	\partial_{z_1}(g_2)	& \ldots 	& \partial_{z_1}(g_s)		\\
\partial_{z_2} & \partial_{z_2}(g_1) & 	\partial_{z_2}(g_2)	& \ldots 	& \partial_{z_2}(g_s)		\\
\vdots  & \vdots & & & \vdots  \\	
\partial_{z_{n-1}} & \partial_{z_{n-1}}(g_1) & 	\partial_{z_{n-1}}(g_2)	& \ldots 	& \partial_{z_{n-1}}(g_s)		\\
\partial_{y} &  \partial_{y}(g_1) & 	\partial_{y}(g_2)	& \ldots 	& \partial_{y}(g_s)
	\end{array}\right)
\ee
Let $\matE$ be the row vector $( \partial_{y}(g_1) \;	\partial_{y}(g_2) \; \ldots \; \partial_{y}(g_s))$ and let $\matD$ be the $(n-1) \times s$ submatrix of $\Jac(g_1,\ldots,g_s)$ obtained by removing $\matE$ from $\Jac(g_1,\ldots,g_s)$ so that
\be
\Jac(I) = \left( \begin{array}{c}
\matD \\
\hline
\matE
\end{array}\right).
\ee
By the chain rule, the Jacobian matrix of $I^\homog$ is
\be
\Jac(G_1,\ldots,G_s) = \left( \begin{array}{c}
\matD^\homog \\
\partial_{w_{0,1}}(F) \cdot \matE^\homog\\
\partial_{w_{0,2}}(F) \cdot \matE^\homog\\
\vdots\\
\partial_{w_{n,d}}(F) \cdot \matE^\homog
\end{array}\right),
\ee
where $\matD^\homog$ and $\matE^\homog$ are obtained by applying $(\Arg)^\homog$ to every entry of $\matD$ and $\matE$,  and $\partial_{y_i}(F) \cdot \matE^\homog$ is the scalar product of $\partial_{y_i}(F)$ and $\matE^\homog$.\\

\noindent{\bf Claim.} One has
\be
I_c(\Jac(I^\homog)) = I_c(\matD^\homog) + (\partial_{w_{0,1}}(F),\ldots,\partial_{w_{n,d}}(F)) \cdot I_c((\Jac(I))^\homog).
\ee

{\em Proof of Claim.} Write $\matE^\homog=(e_1,\ldots,e_s)$. Let $H$ be a $c$-minor of $\Jac(I^\homog)$. Observe that if $H$ is obtained by taking at least two of the last $t$ rows. In particular, we have 
\begin{align*}
H &= \det \left( \begin{array}{cccc}
\vdots & && \vdots \\
\partial_{w_{j,\ell}}(F) e_1 & \partial_{w_{j,\ell}}(F) e_2 & \ldots & \partial_{w_{j,\ell}}(F) e_s\\
\vdots & && \vdots \\
\partial_{w_{j',\ell'}}(F) e_1 & \partial_{w_{j',\ell'}}(F) e_2 & \ldots & \partial_{w_{j',\ell'}}(F) e_s\\
\vdots & && \vdots \\
\end{array}\right) \\
&= \partial_{w_{j,\ell}}(F) \partial_{w_{j',\ell'}}(F) \det\left( \begin{array}{cccc}
\vdots & && \vdots \\
 e_1 & e_2 & \ldots & e_s\\
\vdots & && \vdots \\
 e_1 & e_2 & \ldots & e_s\\
\vdots & && \vdots \\
\end{array}\right) =0.
\end{align*}
Therefore, every non-zero $c \times c$ minor  of $\Jac(I^\homog)$ involves at most one of the last $(n+1)d$ rows, equivalently, 
\bee\label{I_c}
I_c(\Jac(I^\homog)) = \sum_{j=0}^n \sum_{\ell = 1}^{d} I_c\left( \begin{array}{c}
\matD^\homog \\
\partial_{w_{j,\ell}}(F)\cdot \E^\homog
\end{array}\right).
\eee
Observe that if $H$ is a $c \times c$ minor  of $\left( \begin{array}{c}
\D^\homog \\
\partial_{w_{j,\ell}}(F)\cdot \matE^\homog
\end{array}\right)$ not involving   the last row, then $H\in I_c(\matD^\homog)$, while if $H$ involves the last row of the above matrix, then $H= \partial_{w_{j,\ell}}(F) \cdot \det (\Theta)$, where $\Theta$ is a $c$ by $c$ submatrix of $\left( \begin{array}{c}
\matD^\homog \\
\matE^\homog
\end{array}\right)$ that involves the last row. Since $\partial_{w_{j,\ell}}(F) I_c(\matD^\homog) \subseteq I_c(\matD^\homog)$, we can write 

\be
I_c\left( \begin{array}{c}
\matD^\homog \\
\partial_{w_{j,\ell}}(F)\cdot \matE^\homog
\end{array}\right) = I_c(\matD^\homog) +\partial_{w_{j,\ell}}(F) \cdot I_c\left( \begin{array}{c}
\matD^\homog \\
\matE^\homog
\end{array}\right).
\ee
Substituting the above in Equation~\eqref{I_c} for every $i$, we obtain
\be
\begin{array}{ll}
I_c(\Jac(I^\homog))  & = I_c(\matD^\homog) + \sum_{j=0}^n \sum_{\ell=1}^d \left(\partial_{w_{j,\ell}}(F) \cdot I_c\left( \begin{array}{c}
\matD^\homog \\
\matE^\homog
\end{array}\right) \right)  \\			& =  I_c(\matD^\homog) + (\partial_{w_{0,1}}(F), \ldots,\partial_{w_{n,d}}(F)) \cdot I_c\left( \begin{array}{c}
\matD^\homog \\
\matE^\homog
\end{array}\right) \\
			& =  I_c(\matD^\homog) + (\partial_{w_{0,1}}(F), \ldots,\partial_{w_{n,d}}(F)) \cdot I_c((\Jac(I))^\homog)
\end{array}
\ee
proving the claim. 

Let $\Min(I_c(\Jac(I))=\{\p_1,\ldots,\p_r\}$ be the minimal primes in $T$ of $\Sing(I)$. By Lemma \ref{Q^h}(2) each $\p_i^{\homog}$ is prime. We claim that $\{\p_i^{\homog}\,\mid\,i=1,\ldots,r\}$ are the minimal primes of $I_c(\Jac(I^{\homog}))$ of height at most $n$. To this end, we first observe that $I_c(\Jac(I^{\homog}))\subseteq \p_i^{\homog}$ -- this follows from the claim and the fact that $\p_i^{\homog}$ contains both $I_c(D^{\homog})$ and $I_c(\Jac(I)^{\homog})$.

Next, we show that any prime containing $I_c(\Jac(I^{\homog})$ has either height at least $n+1$ or it contains one of the $\p_i^{\homog}$. This will conclude the proof.

So, let $\q$ be a minimal prime ideal with $ I_c(\Jac(I^\homog)) \subseteq \q$.  By the claim,
\[(\partial_{w_{0,1}}(F), \ldots,\partial_{w_{n,d}}(F)) \cdot I_c((\Jac(I))^\homog) \subseteq \q.\]
As $(\partial_{w_{0,1}}(F), \ldots,\partial_{w_{n,d}}(F)) \subseteq \q$, by Lemma~\ref{partials}, $\h(\q)\geq n+1>\dim(T)$.  If $(\partial_{w_{0,1}}(F), \ldots,\partial_{w_{n,d}}(F)) \not \subseteq \q$, then, since $\q$ is prime, the ideal $\q$ contains $(I_c(\Jac(I)))^\homog  = I_c((\Jac(g_1,\ldots,g_s))^\homog)=\p_1^{\homog}\cap \p_2^{\homog}\cap \cdots \cap \p_r^{\homog}$, where the rightmost equality follows by Lemma \ref{Q^h}(4). Then $\q$ contains one of the $\p_i^{\homog}$. It follows that each of the $\p_i^{\homog}$ is a minimal prime of $I_c(\Jac(I^{\homog}))$ and these are the only minimal primes of height at most $n=\dim(T)$.

When ${\rm char}(k)=p>0$ and $p$ does not divide the degree of a minimal generator of $I$, the Jacobian criterion states that the singular locus of $I$ is defined, up to radical, by $I+I_c(\Jac(I))$. The proof follows by a similar argument with the following differences: let $\{\p_1,\ldots,\p_r\}$ be the minimal primes of $I+I_c(\Jac(I))$; in the last part of the proof, we let $\q$ be a prime ideal containing $I^{\homog} + I_c(\Jac(I^\homog))$, and after finding that $(I_c(\Jac(I)))^\homog \subseteq \q$ we have $$I^{\homog} + (I_c(\Jac(I)))^\homog = (I+I_c(\Jac(I)))^\homog\subseteq \q,$$ thus $\q$ contains a minimal prime of $(I+I_c(\Jac(I)))^\homog$. 

\end{proof}

We now apply the preceding theorem to the defining prime ideal of the Rees-like algebra of a homogeneous ideal. Combining it with Theorem~\ref{thm:MinPrimeJacobian} we obtain:

\begin{Corollary}\label{RLhomog} Let $\RLP(I) \subset T$ be the defining prime ideal of $\R(I)$ for some homogeneous ideal $I \subset S = k[x_1,\ldots,x_n]$ and suppose $k$ is algebraically closed and either ${\rm char}(k)=0$ or ${\rm char}(k)=p>0$ does not divide the degree of any minimal generator of $I$.  

Using the \straightensp from Construction~\ref{primeH}, $\RLP(I)^\shomog$ is a nondegenerate, homogeneous prime ideal in a standard graded polynomial ring $T^\shomog$  which defines a projective variety $X$ such that $\codim_X(\Sing X) =\h(I)$.  
\end{Corollary}

\section{Application: Smooth hyperplane sections}

It is natural to ask if Rees-like algebras and \straighteningsp are sufficient to give a smooth counterexample to 
the Eisenbud-Goto conjecture.
We exploit the work so far to settle this in the negative, giving further evidence for Equation~\eqref{eq:CMreg} in the smooth case. More precisely, we show that a nonzero, homogeneous ideal $I \subset S$ is Cohen-Macaulay if and only if a prime \straightensp of its Rees-like algebra $\R(I)$, which preserves the codimension of the singular locus has a hyperplane section that is both smooth and preserves the original graded Betti numbers.  The rest follows by giving a sufficient bound on the regularity of Cohen-Macaulay ideals. For simplicity of exposition, the reader may focus only on the prime \straightensp from Construction~\ref{primeH}. 

\begin{Theorem}\label{smooth} 
Let $k$ be a field with ${\rm char}(k)=0$, let $S=k[x_1,\ldots,x_n]$, and let $I$ be a proper homogeneous $S$-ideal. Let $X\subseteq \mathbb P^N$ denote the projective variety corresponding to the prime standardization from Construction~\ref{primeH} applied to the Rees-like algebra of $I$.  The following two conditions are equivalent:
\begin{enumerate}[(i)]
\item There exists a regular sequence of general hyperplane sections of $X$ such that the resulting variety is smooth;
\item $S/I$ is Cohen-Macaulay.
\end{enumerate}
\end{Theorem}

\begin{proof} \sloppy Set $I = (f_1,\ldots,f_m) \subset S = k[x_1,\ldots,x_n]$.  
Let $\RLP(I)$ be the defining prime ideal of $T$ and let $T \to T^{\shomog}$ be the prime \straightensp defined in Construction~\ref{primeH}. By Bertini's theorem (cf. \cite{Fl}), we may factor out a regular sequence of $\depth(T^{\shomog}/\RLP(I)^\homog)-1$ general linear forms and preserve both the graded Betti numbers of $\RLP(I)^\homog$ and primeness.  Doing so reduces both the dimension of the associated projective variety and that of its singular locus by $\depth(T^{\shomog}/\RLP(I)^\homog)-1$.  Thus one obtains a smooth variety if and only if one has 
\[
\depth(R/\RLP(I)^\homog)-1 > \dim(\Sing \Proj (T^{\shomog}/\RLP(I)^\shomog)),
\]
or equivalently
\[
\dim(T^\homog) - \depth(R/\RLP(I)^\homog)+1 < \h(J)+ 1,\]
where $J$ is the defining ideal of $\Sing \Proj (T^{\shomog}/\RLP(I)^\shomog)$ in $T^{\homog}$.
By Corollary~\ref{RLhomog}, $\h(J) = m + \h(I)$. 
By the Auslander-Buchsbaum theorem, Theorem~\ref{RLbasics} and Proposition~\ref{Q^h} one has
$$
\dim(T^\homog) - \depth(R/\RLP(I)^\homog) = \pd(R/\RLP(I)^{\shomog}) = \pd(S/I) + m - 1.
$$  Thus, the above inequality holds if and only if 
$$\pd(S/I) + m < \h(I) + m + 1$$
or equivalently $\pd(S/I) \le \h(I),$
which occurs if and only if $S/I$ is Cohen-Macaulay.

\end{proof}

We recall that among all Cohen-Macaulay ideals $I$ generated by forms of fixed degrees, complete intersections have the largest regularity.

\begin{Lemma} \label{CMreg}$[$c.f. Huneke et. al. \cite[3.1]{HMNU}$]$ Let $S = k[x_1,\ldots, x_n]$ and $I$ be a homogeneous $S$-ideal such that $S/I$ is Cohen-Macaulay.  If $d_i = \deg(f_i)$, then $\reg(S/I) \le \sum_{i = 1}^m (d_i - 1)$.
\end{Lemma}

The main result of this section depends on the following elementary lemma whose proof is left to the reader. 

\begin{Lemma}\label{ineq} Let $d_1,\ldots,d_m$ be positive integers,  $$\sum_{i = 1}^m d_i \le \prod_{i = 1}^m (d_i + 1) -m.$$

\end{Lemma}

Here we show that any of the smooth hyperplane sections of Rees-like varieties described above satisfy the Eisenbud-Goto Conjecture \cite{EG} giving further evidence that it remains true for smooth varieties.

\begin{Corollary}\label{ACMBertini} Let $k$ be a field with ${\rm char}(k)=0$. Let $I = (f_1,\ldots,f_m) \subset S = k[x_1,\ldots,x_n]$ be a homogeneous ideal such that $S/I$ is Cohen-Macaulay.  Any of the smooth hyperplane sections of the prime \straightensp from Construction~\ref{primeH} applied to the Rees-like prime of $I$ described above satisfies Equation~\eqref{eq:CMreg}.

\end{Corollary}

\begin{proof} \sloppy Set $d_i = \deg(f_i)$ and set $\overline{T^{\homog}}$ to be the quotient of $T^\homog$ by $\depth( T^{\shomog}/\RLP(I)^{\shomog}) - 1$ general linear forms. Similarly set $\overline{\RLP(I)^{\shomog}}$ to be $\RLP(I) \overline{T^{\shomog}}$.  If $m = 1$, then $\overline{\RLP(I)^\homog}$ is a hypersurface and the claim holds.
  If $m \ge 2$, then
  \begin{align*} & \phantom{=\,\, } \reg(\overline{T^{\shomog}}/\overline{\RLP(I)^\homog}) \\
  &= \reg(T^\homog/\RLP(I)) \\
  &= \reg(S/I) + 2 + \sum_{i = 1}^m d_i & \text{ by Theorem~\ref{RLbasics}}\\
  &\le \sum_{i = 1}^m (d_i - 1) + 2 + \sum_{i = 1}^m d_i& \text{ by Lemma~\ref{CMreg}}\\
  & \le 2 \sum_{i = 1}^m d_i & \text{ since } m \ge 2\\
  & \le 2 \prod_{i = 1}^m (d_i + 1) - m & \text{ by Lemma~\ref{ineq}}  \\
    & = \deg(T^{\shomog}/\RLP(I)^\homog) - \h(\RLP(I)^\homog) & \text{ by Theorem~\ref{RLbasics}}\\
  & = \deg(\overline{T^{\shomog}}/\overline{\RLP(I)^\homog}) - \h(\overline{\RLP(I)^\homog}).
\end{align*}
\end{proof}

\section{Seminormality and Weak Normality}\label{sec:Seminormal}

Rees-like algebras are domains, hence they satisfy Serre's conditions $(R_0)$ and $(S_1)$. However, it is easy to check that they are never normal (see Proposition \ref{notnormal} below). 
When $I=(f)$ is a hypersurface,  $\RLP(I)=(y^2-zf^2)$ fails Serre's condition $(R_1)$, however it satisfies Serre's condition $(S_i)$ for all $i$. 

In contrast, we show that whenever $\h(I)>1$, the ideal $\RLP(I)$ satisfies $(R_1)$ but not $(S_2)$. First, let us recall the following equivalent form of Theorem \ref{thm:MinPrimeJacobian}.

\begin{Theorem}\label{R_bla} 
Let $k$ be a field with ${\rm char}(k) = 0$ and let $S$ be a polynomial ring over $k$.
For any nonzero, proper ideal $I \subset S$, the Rees-like algebra $\R(I)$ satisfies Serre's condition $(R_{h -1})$, where $h = \h(I)$, and does not satisfy Serre's condition $(R_h)$.
\end{Theorem}

\begin{Proposition}\label{notnormal} For any nonzero, proper ideal $I \subset S$, the Rees-like algebra $\R(I)$ is not normal.  
\end{Proposition}

\begin{proof}
Since $\R(I)$ is a domain, we show that $\R(I)$ is not integrally closed in its field of fractions. For any 
$0\neq f\in I$ we have $t = \frac{ft^2}{ft} \in \Frac (\R(I))$, and it follows that $\Frac \R(I) = S(t) = \Frac (S[t])$.
  Clearly $t \notin \R(I)=S[It,t^2]$ and $t$ satisfies the monic polynomial equation $X^2 - t^2 \in \R(I)[X]$. Then $\R(I)$ is not integrally closed.  
\end{proof}

\begin{Corollary}\label{S_2}
If $\h(I)>1$, then $\R(I)$ does not satisfies Serre's condition $(S_2)$.
\end{Corollary}

 We turn our attention then to alternate forms of normality, namely weak normality and seminormality. We quickly review these notions, but for a more thorough treatment, consult \cite{Vi}. 

\begin{Definition}
For a finite extension $A \subset B$ of reduced rings. A subextension $A \subset C \subset B$ is {\bf subintegral} provided it is integral, induces a bijection on spectra, and an isomorphism on residue fields at all points. It is called {\bf weakly subintegral} provided one only asks for purely inseparable extensions of residue fields. 
\end{Definition}

In any extension $A \subset A^{\textrm{N}}$ of a ring into its normalization, there is a unique largest subextension $A \subset A^{\textrm{SN}} \subset A^{\textrm{N}}$ which is subintegral and one says that $A$ is {\bf seminormal} provided that $A = A^{\textrm{SN}}$. Similarly, there is a unique largest subextension which is weakly subintegral $A \subset A^{\textrm{WN}} \subset A^{\textrm{N}}$  and we say that $A$ is {\bf weakly normal} if $A = A^{\textrm{WN}}$. Consequently all weakly normal rings are seminormal and all seminormal rings are weakly normal. 

A prototypical example of a seminormal ring which is not normal is the pinch point $k[x,xt,t^2] \cong k[x,y,z]/(y^2 - zx^2)$, where $\mathrm{char}(k) \neq 2$. This ring corresponds to the Rees-like algebra a single linear form. We show that quite often, Rees-like algebras are seminormal and weakly normal.  To do this, we exploit the following useful  criteria.

\begin{Theorem}\label{thm:criteriaSNWN} For a reduced ring $A$, 
\begin{enumerate}
\item \cite[Prop. 1.4]{LeVi} $A$ is seminormal if and only if for a  fixed pair of relatively prime integers $0 < r <s$, when $b \in A^{\textrm{N}}$ satisfies $b^r \in A$ and $b^s \in A$ then $b \in A$,
\item \cite[Thm~1]{Ya} if the characteristic of $A$ is $p > 0$, then $A$ is weakly normal if for each $b \in A^{\textrm{N}}$ such that $b^p \in A$ then $b \in A$. 
\end{enumerate}
\end{Theorem}

For the remainder of this section, set $S = k[x_1,\ldots,x_n]$ a polynomial ring and $I$ a homogeneous ideal in $S$. Our next goal illustrates the general theme of characterizing geometric properties of the Rees-like algebra of $I$ in terms of algebraic properties of $I$. Recall the normalization of $\R(I)$ is $S[t]$.

\begin{Theorem}\label{thm:RLseminormal}
With the notation as above, the following are equivalent:
\begin{enumerate}
\item $I$ is radical, 
\item for every odd integer $\sigma>1$, for every $b \in S[t]$, if $b^{\sigma} \in \R(I)$ then $b \in \R(I)$, 
\item there are two coprime integers, $r$ and $s$ both greater than $1$ such that for every $b \in S[t]$, if $b^r \in \R(I)$ and $b^s\in \R(I)$, then $b \in \R(I)$.
\item there is an odd integer $\sigma>1$ such that for every $b \in S[t]$, if $b^{\sigma} \in \R(I)$ then $b \in \R(I)$, 
\end{enumerate} 
\end{Theorem}

\begin{proof}
The implications (2) $\Lra$ (3) $\Lra$ (4) are clear. We first prove (4) $\Lra$ (1). Assume $a\in S$ and $a^n \in I$ for some $n\in \ZZ_+$. If $r\in \ZZ_+$ with $\sigma^r\geq n$, then $a^{\sigma^r}\in I$. Thus $(at)^{\sigma^r}\in \R(I)$. By assumption (3) it follows that $at\in \R(I)=S[It,t^2]$, and so $a\in I$.

The theorem follows by showing that (1) $\Lra$ (2). Fix an odd integer $\sigma>1$ and assume $I$ is radical. Let $b\in S[t]$ be an element such that $b^{\sigma} \in \R(I)$, we need to show that $b\in \R(I)$. We consider the grading on $S[t]$ given by $\deg(t)=1$ and $\deg(f)=0$ for every $f\in S$. 
Write $b=\sum_{j=1}^{r}b_jt^{i_j}$ with $b_j\in S$, for integers $0\leq i_1<i_2<\ldots<i_s$ and elements $b_j\in S$.\\

\noindent {\bf Claim.} We may assume $b_j\notin I$ for any $j$.\\

To prove the claim,  observe that if $d\in I$, then $dt^k\in \R(I)=S[It,t^2]$ for every $k\geq 1$. 
Now, assume $b_j\in I$ for some $j$. Expand
\be
(b-b_jt^{i_j})^{\sigma} = b^{\sigma} + \sum_{h=1}^{\sigma} \binom{\sigma}{h}b^{\sigma - h}b_j^h
t^{hi_j}.
\ee
Since $b_j \in I$, each  $\binom{\sigma}{h}b^{\sigma - h}b_j^h \in I$ and thus $ \binom{\sigma}{h}b^{\sigma - h}b_j^h t^{hi_j}\in S[It,t^2]$ for every $1\leq h \leq \sigma$. It follows that $b^{\sigma} \in S[It,t^2]$ if and only if $(b-b_jt^{i_j})^{\sigma}\in S[It,t^2]$. We may then decompose $b=\widetilde{b} + \widetilde{c}$ where $\widetilde{b} = \sum b_jt^{i_j}$ with each $b_j\notin I$ and $\widetilde{c} = \sum b_jt^{i_j}$ with each $b_j\in I$. By the above, $b^{\sigma}\in S[It,t^2]$ if and only if $\widetilde{b}^{\sigma} \in S[It,t^2]$ and $b\in S[It,t^2]$ if and only if $\widetilde{b}\in S[It,t^2]$. By replacing $b$ by $\widetilde{b}$ we can assume $b_j\notin I$ for any $j$, proving the claim.\\

It suffices to show that each $i_j$ is even, because then $b\in S[t^2] \subset S[It, t^2]$. We proceed by induction on the number $r\geq 1$ of homogeneous components of $b$. 
If $r=1$, then $b=b_1t^{i_1}$. Assume by contradiction that $i_1$ is odd. Since $b^{\sigma}=b_1^{\sigma}t^{i_1\sigma}\in S[It,t^2]$ and $i_1\sigma$ is odd, then $b_1^{\sigma}\in I$. Since $I$ is radical, this implies $b_1\in I$, yielding a contradiction. Therefore $i_1 \in 2\ZZ$. 

Next, assume $r>1$. Assume by contradiction one of the $i_j$ is odd, we let $u=\min\{j\,\mid\, i_j \text{ is odd }\}$. Observe that $b_1^{\sigma-1}b_u t^{i_1(\sigma-1)+i_u}$ is the homogeneous component of smallest odd degree of $b^{\sigma}\in S[It,t^2]$, thus it lies in $S[It,t^2]$. Since $i_1(\sigma-1)+i_u$ is odd, then 
$b_1^{\sigma-1}b_u \in I$ and so $(b_1b_u)^{\sigma-1}\in I$. Since $I$ is radical, we obtain $b_1b_u\in I$.
Now consider
\be
d=b_ub = b_1b_ut^{i_1} + b_2b_ut^{i_2} + \ldots +b_u^2t^{i_u} + \ldots
\ee
Set $e:=d-b_1b_ut^{i_1}$.  Since $b_1b_u\in I$, by the proof of the claim it follows that $e^{\sigma}\in S[It,t^2]$. By induction, it follows that $e\in S[It,t^2]$. Since $b_1b_ut^{i_1}\in S[It,t^2]$ too, then $d\in S[It,t^2]$, so every homogeneous component of $d$ lies in $S[It,t^2]$. In particular, the homogeneous component of degree $i_u$, i.e. $b_u^2t^{i_u}$ lies in $S[It,t^2]$. Since $i_u$ is odd, then $b_u^2\in I$, since $I$ is radical, $b_u\in I$ which is a contradiction.

\end{proof}

\begin{Remark} One must work with odd integers in Theorem~\ref{thm:RLseminormal}. 
If $\sigma$ is even and $b=t\in S[t]$ one has $b^2\in S[It,t^2]$ but $b\notin S[It,t^2]$.

\end{Remark}

Combining Theorems~\ref{thm:criteriaSNWN} and \ref{thm:RLseminormal}, one has the following immediate corollary.  

\begin{Corollary}\label{cor:criteriaseminormal}
Let $k$ be a field with $\rm{char}(k) \neq 2$ and let $S$ a polynomial algebra over $k$. A homogeneous ideal $I$ is radical if and only if its Rees-like algebra $\R(I)$ is seminormal which happens if and only if $\R(I)$ is weakly normal. 
\end{Corollary}

One should notice that the analogous statement for Rees algebras does not hold. 
Indeed, the following is an example of a radical ideal $I$ whose Rees Algebra $R[It]$ is not seminormal. This example was found with the help of the Macaulay2 Seminormalization package of   Serbinowski and Schwede \cite{SS, M2}.

\begin{Example} Let $k$ be a field and $S = k[x,y,z]$. Let $\p = (y^4 - x^3z, xy^3 - z^3, x^4 - yz^2)$ be the ideal defining the monomial curve $k[v^9, v^{10}, v^{13}]$. Then \[\p = I_2\begin{pmatrix} z & -y & x \\
-y^3 & x^3 & -z^2
\end{pmatrix}.\]
By \cite[p. 309]{Vasconcelos}, $\p$ is not normal; that is, not all powers of $\p$ are integrally closed and thus the  Rees algebra $\mathcal{R}(\p) = S[\p t]$ is not a normal ring.  

Write $p_1 = y^4 - x^3z$, $p_2 = xy^3 - z^3$, and $p_3 = x^4 - yz^2$. Now set 
$$
f = \frac{x^2(p_2t)(p_3t) + z (p_1t)^2}{y} = (x^7y^2 - 3x^3y^3z^2 + x^2z^5+y^7z)t^2 \in \mathrm{Frac}(S[\p t]).
$$
Since no product of two monomial terms among the generators of $\p$ divides $x^2z^5$, it follows that $f \notin S[\p t]$.  However, we verify below that $f^2, f^3 \in S[\p t]$. Indeed,
\[f^2 = (-y p_1^3p_3 + x^2p_1p_3^3+yp_2^4-zp_2^3p_3+xzp_3^4)t^4 \in \p^4t^4 \subseteq S[\p t],\]
and
\[f^3 = (-z p_1^5p_3 + z p_1^2p_2^4  + 3xz p_1^2p_2p_3^3 + z^2 p_1p_2p_3^4 + x^3 p_2^2p_3^4) t^6 \in \p^6t^6 \subseteq S[\p t].\]
By Theorem~\ref{thm:criteriaSNWN}, we see that $S[\p t]$ is not seminormal. However, since $\p$ is prime, $\R(\p)=S[\p t,t^2]$ is seminormal  by Corollary \ref{cor:criteriaseminormal}.

\end{Example}

In positive characteristic, $F$-split rings are weakly normal, so in view of Corollary \ref{cor:criteriaseminormal} one may hope to find a fairly large class of ideals $I$ for which $\R(I)$ is $F$-split. As such from this point forward, for simplicity, we fix a perfect ground field $k$ and all rings and fields considered for the rest of this section are $F$-finite. We also identify the Frobenius map with the inclusion $S \subset S^{1/p}$ into a choice of $p$-th roots of elements of $S$ from a fixed algebraic closure.

\begin{Theorem}\label{thm:Fsplit}
Suppose ${\rm char}(k)=p > 2$ and $I$ is a radical ideal in $S$. The ring $S/I$ is $F$-split if and only if $\R(I)$ is $F$-split.
\end{Theorem}

\begin{proof} Assume $S/I$ is $F$-split. Every splitting of $S/I$ is induced by a splitting $\varphi \colon S^{1/p} \to S$ of $S$ with $\varphi(I^{1/p}) \subset I$. Next, we consider $\R(I) = S[It, t^2]$ as a graded subring of $S[t]$.  Define $\psi \colon S[t]^{1/p} \to S[t]$ by writing $f \in S[t]^{1/p}$ as $f = \sum a_i^{1/p} t^{i/p}$ and setting 
$$\psi(f) = \!\!\!\!\!\!\sum_{i \,\equiv\, 0 \,\,({\rm mod \,} p)} \!\!\!\!\!\!\varphi(a_i^{1/p}) t^{ \frac{i}{p}}.$$
Clearly this is $S$-linear and $\psi( t\cdot f) = t \psi( f)$ for each  $f \in S[t]^{1/p}$. Thus $\psi$ is $S[t]$-linear, whence $\R(I)$-linear. Moreover $\psi$ is surjective because $\psi(1)=1$. We show that the $\psi(\R(I)) \subseteq \R(I)$. This will show that $\psi|_{\R(I)}$ is an $F$-splitting of $\R(I)$. Let $f = \sum a_i^{1/p} t^{i/p} \in \R(I)^{1/p}$, so $a_i \in S$ for every even $i$ and $a_i \in I$ for every $i$ odd. To prove $\psi(f) \in \R(I)$ we need to show that if $\frac{i}{p}$ is an odd integer, then $ \varphi(a_i^{1/p})\in I$. This follows since $\frac{i}{p}$ being odd implies that $i$ is odd. Thus we have $a_i^{1/p} \in I^{1/p}$ and so $\varphi(a_i^{1/p}) \in I$.

Conversely, assume $\R(I)$ is $F$-split. We may assume without loss of generality that $\psi \colon \R(I)^{1/p} \to \R(I)$ is a splitting which is graded of degree $0$. Denote by $\psi_0 \colon S^{1/p} \to S$ the restriction of $\psi$ to the degree $0$ part of $\R(I)$. This is clearly $S$-linear and surjective, so it suffices to see that $\psi_0(I^{1/p}) \subset I$. By $\R(I)$-linearity, for $a \in I$ we have 
$$at^2 \psi_0( a^{1/p}) = \psi( a^{1/p} \cdot at^2) = at \psi( a^{1/p}t).$$ As $\R(I)$ is a domain, we have $ \psi_0( a^{1/p})t = \psi( a^{1/p}t)$. Since $\psi$ is graded, $\psi( a^{1/p} t) \in It$, so $\psi_0( a^{1/p})t\in It$ and then $\psi_0( a^{1/p} ) \in I$, as desired.

\end{proof}

\section{Canonical Module of a Rees-like Algebra}\label{canonical}

In this section we give an explicit computation of the canonical module of the Rees-like algebra of an arbitrary ideal in a polynomial ring. We also give an explicit free resolution of the canonical module over the presenting polynomial ring $T$.  This resolution has a surprising structure obtained by combining two Koszul complexes.  

For simplicity, we assume $k$ is a field with $\mathrm{char}(k) \neq 2$.  Let $S = k[x_1,\ldots,x_n]$ and $f_1,\ldots,f_m$ minimal generators of a homogeneous ideal $I$.  We also assume that $\h(I) \ge 2$.  Denote by $\mathcal{RLP}(f_1,\ldots,f_m)$ the Rees-like prime defined in Section 2. 
There is a distinguished complete intersection in $\mathcal{RLP}(f_1,\ldots,f_m)$, namely, 
\[C = \left(y_1^2 - zf_1^2, y_2^2 - z f_2^2,\ldots, y_m^2 - zf_m^2 \right).\] We compute the canonical module by exploiting linkage theory relative to $C$.  First we compute its primary decomposition.
  
\sloppy Recall that the choice of different minimal generating sets of $I$ give different but isomorphic Rees--like primes in the same polynomial ring $T=S[y_1,\ldots,y_m,z]$. For instance, if ${\rm char}(k)\neq 2$, then $\RLP(f_1,-f_2,f_3,\ldots,f_m)\neq \RLP(f_1,\ldots,f_m)$.
  
\begin{Lemma} \label{primdec}
With the the notation above, we have the following:
\begin{enumerate}
\item For any choice of $\mypm$ signs, $C \subset \mathcal{RLP}(\mypm f_1, \mypm f_2, \ldots, \mypm f_m)$.
\item $\mathcal{RLP}(f_1, f_2, \ldots, f_m) = \mathcal{RLP}(- f_1, - f_2, \ldots, - f_m)$.
\item If $m \ge 2$, then for any choice of $\mypm$ sign as indicated
\[\mathcal{RLP}(f_1, f_2, \ldots, f_m) \neq \mathcal{RLP}(f_1, - f_2, \mypm f_3, \mypm f_4, \ldots, \mypm f_m).\]
\item The complete intersection ideal $C$ defined above is radical and has the following primary decomposition
\[ C = \bigcap \mathcal{RLP}(f_1, \mypm f_2, \mypm f_3, \ldots, \mypm f_m),\]
where the intersection is taken over all possible choices of $\mypm$ sign.
\end{enumerate}

\end{Lemma}

\begin{proof}
(1) One simply observes that when we replace $y_i$ by $\mypm f_it$ and $z$ by $t^2$, we see that $y_i^2 - zf_i^2$ becomes $(\mypm f_it)^2 - t^2 f_i^2 = 0$.\\
(2) Let $\phi: T \to S[t]$ be the map sending $y_i \mapsto f_i t$ and $z \mapsto t^2$.  Then clearly $\mathcal{RLP}(f_1, f_2, \ldots, f_m) = \ker(\phi) = \ker(-\phi) = \mathcal{RLP}(-f_1, -f_2, \ldots, -f_m)$.\\
(3) The element $y_1y_2 - zf_1f_2$ is in the left-hand ideal but not the right-hand one.\\
(4) By (3), there are $2^{m-1}$ distinct primes in the intersection above, let us write them $Q_1,\ldots,Q_{2^{m-1}}$.  By (1), $C$ is a subset of the ideal $H=\bigcap_{j=1}^{2^{m-1}}Q_j$. 
Both $C$ and $Q$ are unmixed homogeneous ideals with the grading $\deg(x_j)=1$, $\deg(y_i)=d_i+1$ and $\deg(z)=2$. Since $y_i^2 - zf_i^2$ is homogeneous of degree $2(\deg(f_i)+1)$, we have $e(T/C)=2^mD$, where $D=\prod_{i=1}^m(d_i+1)$. By Theorem \ref{RLbasics}, $e(T/Q_i)=2D$ for every $i=1,\ldots,2^{m-1}$. Then $e(T/C)=e(T/H)=2^mD$.  Since $C\subseteq H$ are unmixed ideals of the same multiplicity and height, then $C = H$.

\end{proof}

Next, we want to obtain an explicit description of the link $L=C:\RLP(I)$, where $\RLP(I)=\mathcal{RLP}(f_1, f_2, \ldots, f_m)$. 
For $1 \le j \le m$ define the elements $g_j^{\rm{even}}$ and $g_j^{\rm{odd}}$ as follows: for a subset $S \subseteq \{1,\ldots,j\}$, let $\underline{y}^S$ denote $\prod_{i \in S} y_i$ and set $\overline{S} = \{1,\ldots,j\} - S$.  Then we define two elements of $T$:

\begin{align*}
g_j^{\rm{even}} &\coloneqq \sum_{i = 0}^{\lfloor j/2 \rfloor} \sum_{\substack{S \subseteq \{1,\ldots,j\}\\ |S| = 2i}} \underline{y}^{\overline{S}} \underline{f}^{S} z^i,\\
g_j^{\rm{odd}} & \coloneqq \sum_{i = 0}^{\lfloor (j-1)/2 \rfloor} \sum_{\substack{S \subseteq \{1,\ldots,j\}\\ |S| = 2i+1}}\underline{y}^{\overline{S}} \underline{f}^{S} z^i.\\
\end{align*}
For example, when $m = 4$ we get
\begin{align*}
g_4^{\rm{even}} &= y_1y_2y_3y_4 + y_1y_2f_3f_4z + y_1f_2y_3f_4z+\cdots+ f_1f_2y_3y_4z + f_1f_2f_3f_4z^2,\\
g_4^{\rm{odd}} &= y_1y_2y_3f_4 + y_1y_2f_3y_4 + \cdots + f_1y_2y_3y_4 + y_1f_2f_3f_4z + \cdots + f_1f_2f_3y_4z.\\
\end{align*}

The elements $g_j^{\rm{even}}$ and $g_j^{\rm{odd}}$ satisfy several useful identities, as the following Lemma shows.

\begin{Lemma}\label{identities} For $ 2 \le j \le m$, we have
\begin{align*}
g^{\rm{odd}}_j&= y_j g^{\rm{odd}}_{j-1} + f_jg^{\rm{even}}_{j-1},\\
g^{\rm{even}}_j&= y_j g^{\rm{even}}_{j-1} + zf_jg^{\rm{odd}}_{j-1},\\
y_jg^{\rm{even}}_j&= zf_jg^{\rm{odd}}_j + \left(y_j^2-zf_j^2\right) g^{\rm{even}}_{j-1},\\
f_jg^{\rm{even}}_j&= y_jg^{\rm{odd}}_j - \left(y_j^2-zf_j^2\right) g^{\rm{odd}}_{j-1}.\\
\end{align*}
\end{Lemma}

\begin{proof} We prove the first identity. 
\begin{align*} g^{\rm{odd}}_j&=\sum_{i = 0}^{\lfloor (j-1)/2 \rfloor} \sum_{\substack{S \subseteq \{1,\ldots,j\}\\ |S| = 2i+1}}\underline{y}^{\overline{S}} \underline{f}^{S} z^i.\\
&= \sum_{i = 0}^{\lfloor (j-1)/2 \rfloor} \sum_{\substack{S \subseteq \{1,\ldots,j-1\}\\ |S| = 2i+1}}\underline{y}^{\overline{S}} y_j \underline{f}^{S} z^i
+ \sum_{i = 0}^{\lfloor (j-1)/2 \rfloor} \sum_{\substack{S \subseteq \{1,\ldots,j-1\}\\ |S| = 2i}}\underline{y}^{\overline{S}}  \underline{f}^{S} f_j z^i\\
&=y_j g^{\rm{odd}}_{j-1} + f_jg^{\rm{even}}_{j-1}.\\
\end{align*}

The second identity is proved similarly. 
As for the third identity, we have

\begin{align*}
y_jg^{\rm{even}}_j &= y_j^2 g^{\rm{even}}_{j-1} + y_jzf_jg^{\rm{odd}}_{j-1} \\
&=  y_j^2 g^{\rm{even}}_{j-1} +zf_j\left(g^{\rm{odd}}_j - f_jg^{\rm{even}}_{j-1}\right)\\
& = \left(y_j^2-zf_j^2\right) g^{\rm{even}}_{j-1} + zf_jg^{\rm{odd}}_{j},
\end{align*}
where the first equality follows from the second identity and the middle equality from the first identity.

 The fourth identity is proved similarly.

\end{proof}

\begin{Lemma}\label{evenodd} If $Q = \mathcal{RLP}(f_1,-f_2,\mypm f_3, \mypm f_4,\ldots,\mypm f_m)$, then $g_m^{\rm{even}}, g_m^{\rm{odd}} \in Q$.
\end{Lemma}

\begin{proof} We show for $2 \le j \le m$ that $g_j^{\rm{even}}, g_j^{\rm{odd}} \in Q$ and proceed by induction on $j$.  First note that $g^{\rm{even}}_2 = y_1y_2 + z f_1f_2 =y_1y_2-zf_1(-f_2)\in Q$ and, similarly, $g^{\rm{odd}}_2 =  y_1f_2 + y_2f_1=y_2f_1-y_1(-f_2) \in Q$, 

Now let $j > 2$ and suppose $g^{\rm{even}}_{j-1}, g^{\rm{odd}}_{j-1} \in Q$.  By Lemma~\ref{identities} $g^{\rm{even}}_j = y_j g^{\rm{even}}_{j-1} + zf_j g^{\rm{odd}}_{j-1} \in Q$ and similarly $g^{\rm{odd}}_j = y_j g^{\rm{odd}}_{j-1} + f_j g^{\rm{even}}_{j-1} \in Q$.  We are done by induction.
\end{proof}

\begin{Corollary}\label{evenoddcor} If $Q = \mathcal{RLP}(\mypm f_1, \mypm f_2,\ldots, \mypm f_m)$, then $g_m^{\rm{even}}, g_m^{\rm{odd}} \in Q$ for any choice of $\mypm$ signs except for $Q = \mathcal{RLP}( f_1,\ldots, f_m) = \mathcal{RLP}(- f_1, \ldots, - f_m)$.
\end{Corollary}

\begin{proof} By the symmetry of $g_m^{\rm{even}}, g_m^{\rm{odd}}$, we can assume that the signs on $f_1$ and $f_2$ are different.  Then the statement follows from Lemma~\ref{evenodd}.
\end{proof}

Our next goal is to prove that $C:Q=C+\left(g_m^{\rm{even}}, g_m^{\rm{odd}}\right)$. From now on we adopt the following notation
\begin{Notation}\label{CQL}
 $I=(f_1,\ldots,f_m)\subseteq S$, and $Q = \mathcal{RLP}(f_1,\ldots,f_m)\subseteq T$ is its Rees-like prime;
 $L \coloneqq C:Q\subseteq T$, and 
 $J \coloneqq C+\left(g_m^{\rm{even}}, g_m^{\rm{odd}}\right)\subseteq T$.
\end{Notation} 
Proving $L = J$ will require a sequence of lemmas.  First we construct two useful short exact sequences.  

\begin{Lemma}\label{SES} With Notation \ref{CQL}, we have short exact sequences
\[ 0 \to T/Q \xrightarrow{\cdot g_m^{\rm{odd}}} T/C \to T/(C+(g_m^{\rm{odd}})) \to 0,\]
and
\[0 \to T/(IT + (y_1,\ldots,y_m)) \xrightarrow{\cdot g_m^{\rm{even}}} T/(C+(g_m^{\rm{odd}})) \to T/J \to 0.\]
In particular, $Q = C:(g^{\rm{odd}}_m)$ and $IT + (y_1,\ldots,y_m) = (C + (g^{\rm{odd}}_m)):(g^{\rm{even}}_m)$.
\end{Lemma}

\begin{proof}
The first short exact sequence follows from the fact that $C:(g_m^{\rm{odd}}) = Q$ by Proposition~\ref{primdec} and Corollary~\ref{evenoddcor}.

\noindent Analogously, for the second sequence we need to show $(C + (g^{\rm{odd}}_m)):(g^{\rm{even}}_m) = IT + (y_1,\ldots,y_m)$.  First note that by Lemma~\ref{identities},
\[y_m  g^{\rm{even}}_m= g^{\rm{even}}_{m-1} (y_m^2 - z f_m^2) + z f_m \,g^{\rm{odd}}_m \in C + (g^{\rm{odd}}_m)\]
and
\[f_m g^{\rm{even}}_m= y_m \,g^{\rm{odd}}_m - g^{\rm{odd}}_{m-1} (y_m^2 - z f_m^2)  \in C + (g^{\rm{odd}}_m).\]
By symmetry, it follows that $IT + (y_1,\ldots,y_m) \subseteq (C+(g^{\rm{odd}}_m)):(g^{\rm{even}}_m)$.  Since $IT \subset (C+(g^{\rm{odd}}_m)):(g^{\rm{even}}_m)$, it suffices to consider the reverse inclusion modulo $IT$.  Let $h \in T$ be such that  $h\cdot g^{\rm{even}}_m \in (C+(g^{\rm{odd}}_m)$ modulo $IT$. Since $g^{\rm{even}}_m \equiv y_1y_2\cdots y_m$ modulo $IT$ and $(C+(g^{\rm{odd}}_m)) \equiv (y_1^2,\ldots,y_m^2)$ modulo $IT$, then $hy_1\cdots y_m \in (y_1^2,\ldots,y_m^2)$ in $T/IT$. Since $y_1,\ldots,y_m$ is a regular sequence on $T/IT$, then $h\in (y_1,\ldots,y_m)+IT$. Therefore $IT + (y_1,\ldots,y_m) = (C+(g^{\rm{odd}}_m)):(g^{\rm{even}}_m)$.

\end{proof}

Next we compute the initial ideal of $J$.

\begin{Lemma}\label{init1}
Let $<$ be the lex order $<$ on $T$ and $y_1>y_2>\cdots>y_m>z>x_1>\cdots>x_n$. Then $y_1^2-zf_1^2,\ldots,y_m^2-zf_m^2,g_m^{\rm{even}}, g_m^{\rm{odd}}$ form a Gr\"obner basis of $J$ with respect to $<$.  In particular
\be
\In_<(J) = (y_1^2,\;\ldots,\;y_m^2,\;y_1\cdots y_m, \; y_1\cdots y_{m-1}\In_<(f_m)),
\ee
and $\pd(T/\In_<(J) \leq m+1$.
\end{Lemma}

\begin{proof} For the first part of the statement we show that all $S$-pairs reduce to 0 using the basic identities from Lemma~\ref{identities}. 

Clearly $\In_<(y_i^2-zf_i^2) = y_i^2$ for all $i$.  By \cite[Proposition 2.15]{HerzogEne} the $S$-pairs $S(y_i^2 - z f_i^2, y_j^2 - z f_j^2)$  reduce to $0$ for all $1 \le i < j \leq m$.  
Since $\In_<(g_m^{\mathrm{even}}) = y_1\cdots y_m$, we see that 
\begin{align*}
S(y_m^2 - zf_m^2, g_m^{\mathrm{even}}) &= y_1\cdots y_{m-1} (y_m^2 - zf_m^2) - y_m g_m^{\mathrm{even}}\\
&= y_1\cdots y_{m-1} (y_m^2 - zf_m^2) -  (y_m^2-zf_m^2) g^{\rm{even}}_{m-1} - zf_mg^{\rm{odd}}_{m}\\
&=  (y_1\cdots y_{m-1}  -  g^{\rm{even}}_{m-1}) (y_m^2-zf_m^2) - zf_mg^{\rm{odd}}_{m}.
\end{align*}

Since the two initial terms of $(y_1\cdots y_{m-1}  -  g^{\rm{even}}_{m-1}) (y_m^2-zf_m^2)$ and $zf_mg^{\rm{odd}}_{m}$ are different (One is divisible by $y_m^2$; the other is not.),  this is a standard expression for $h$. Therefore  this $S$-pair reduces to $0$.
By symmetry, $S(y_i^2 - zf_i^2, g_m^{\rm{even}})$ also reduces to $0$ for all $i$.  A similar calculation shows that $S(y_i^2 - zf_i^2,g_m^{\rm{odd}})$ also reduces to $0$ for all $i$.  Finally,  we consider 
\begin{align*}
S(g_m^{\rm{even}},g_m^{\rm{odd}}) &= \In_<(f_m) g_m^{\rm{even}} - y_m g_m^{\rm{odd}}\\
&=  \In_<(f_m) g_m^{\rm{even}} - (f_m g_m^{\rm{even}} + g_{j-1}^{\rm{odd}} (y_m^2 - zf_m^2))\\
&= (\In_<(f_m) - f_m) g_m^{\rm{even}} - g_{j-1}^{\rm{odd}} (y_m^2 - z f_m^2),
\end{align*}
where the second equality follows from Lemma~\ref{identities}.
It is easy to see that last line is a standard expression for $S(g_m^{\rm{even}},g_m^{\rm{odd}})$ and so it also reduces to $0$.

For the second part of the statement we observe that $y_1,\ldots,y_m,a:=\In_<(f_m)$ form a regular sequence; thus $A=k[y_1,\ldots,y_m,a]$ is a polynomial ring in $m+1$ variables. By the first part of the proof, the ideal $\In_<(J)$ is extended from an $A$-ideal, and so $\pd(T/\In_<(J)) \leq m+1$.

\end{proof}

\begin{Remark}
In fact, it is not hard to show that $\pd(T/\In_<(J))=m+1$ and $\beta_{m+1}^T(T/\In_<(T/J))=1$. However, for our intended use of Lemma~\ref{init1}, the inequality $\pd(T/\In_<(J)) \leq m+1$ is sufficient -- see the proof of Proposition \ref{C:Q}.
\end{Remark}

As a step toward proving $J$ is unmixed, we next show $(y_1,y_2,\ldots,y_m,z)$ is not an associated prime of $T/J$. 

\begin{Lemma}\label{nop} 
Let $\p = (y_1,y_2,\ldots,y_m,z)$.  Then $\p \notin \ass(T/J)$.
\end{Lemma}

\begin{proof} First we show that $Q_\p$ is a complete intersection.  Recall that we have a decomposition $Q = \RLP(I)_{\textrm{syz}} + \RLP(I)_{\textrm{gen}}$ as in Theorem~\ref{RLbasics}.  The ideal $\RLP(I)_{\textrm{syz}}$ is generated by elements of the form $\sum_i s_i y_i$ such that $\sum_i s_i f_i = 0$ in $S$.  In particular, the following elements corresponding to Koszul syzygies of $I$ are in $(\RLP(I)_{\textrm{syz}})_\p$: $y_1 - \frac{f_1}{f_m} y_m, y_2 - \frac{f_2}{f_m}y_m,\ldots,y_{m-1} - \frac{f_{m-1}}{f_m} y_m$.  For brevity, set $y_i' = y_i - \frac{f_i}{f_m} y_m$.   Since $y_m^2 - zf_m^2 \in \RLP(I)_{\textrm{gen}}$, it follows that $Q_\p$ is generated by the regular sequence $y_1', y_2', \ldots, y_{m-1}', y_m^2 - zf_m^2$.  (These elements, along with $y_m$, form a regular system of parameters of the regular local ring $S_\p$.)  

Now we compute the link $L_\p = C_\p:Q_\p$.  Set $\overline{y_i} = y_i + \frac{f_i}{f_m} y_m$, so that
\[y_i^2 - z f_i^2 = \overline{y_i} y_i' + \frac{f_i^2}{f_m^2}(y_m^2 - z f_m^2).\]
Therefore 
\[\begin{bmatrix} y_1^2 - z f_1^2,\ldots,y_m^2 - z f_m^2 \end{bmatrix} = \matD \begin{bmatrix} y_1',\ldots,y_{m-1}',y_m^2 - z f_m^2\end{bmatrix}^\mathsf{T},\]
where 
\[\matD = \begin{bmatrix} \overline{y_1} & 0 & \cdots & 0 & f_1^2/f_m^2 \\
0 & \overline{y_2} & \cdots & 0 & f_2^2/f_m^2 \\
0 & 0 & \ddots & 0 & \vdots \\
0 & 0 & 0 & \overline{y_{m-1}} & f_{m-1}^2/f_m^2\\
0 & 0 & 0 & 0 & 1
\end{bmatrix}.\]
By \cite[Theorem A.140]{Vasconcelos2}, $L_p = (C + (\det \matD))_\p$.  Note that
\begin{align*}
\det(\matD) &= \prod_{i = 1}^{m-1} \overline{y_i}\\ 
&= \prod_{i = 1}^{m-1} (y_i + \frac{f_i}{f_m} y_m)\\
&= \sum_{S \subseteq \{1,\ldots,m-1\}} \underline{y}^{\overline{S}} \frac{\underline{f}^S}{f_m^{|S|}} y_m^{|S|}\\
&= \sum_{i = 0}^{\lfloor (m-1)/2 \rfloor} \left( 
\sum_{{\substack{S \subseteq \{1,\ldots,m-1\}\\ |S| = 2i}}} \underline{y}^{\overline{S}} \frac{\underline{f}^S}{f_m^{2i}} y_m^{2i}
+
\sum_{{\substack{S \subseteq \{1,\ldots,m-1\}\\ |S| = 2i+1}}} \underline{y}^{\overline{S}} \frac{\underline{f}^S}{f_m^{2i+1}} y_m^{2i+1}
\right)\\
&\equiv \sum_{i = 0}^{\lfloor (m-1)/2 \rfloor} \left( 
\sum_{{\substack{S \subseteq \{1,\ldots,m-1\}\\ |S| = 2i}}} \underline{y}^{\overline{S}} \underline{f}^S z^i
+
\sum_{{\substack{S \subseteq \{1,\ldots,m-1\}\\ |S| = 2i+1}}} \underline{y}^{\overline{S}} \underline{f}^S z^i \frac{y_m}{f_m}
\right)   (\mathrm{mod} \,\,C_\p)  \\
&= g_{m-1}^{\rm{even}} + \frac{y_m}{f_m} g_{m-1}^{\rm{odd}},
\end{align*}
where the third line follows from expanding the product, the fourth line separates the even and odd terms, and the fifth line follows since $z - \frac{y_m^2}{f_m^2} \in C_\p$.  Finally note that
\[f_m \det(\matD) \equiv f_m g_{m-1}^{\rm{even}} + y_m g_{m-1}^{\rm{odd}} \equiv g_m^{\mathrm{odd}} (\mathrm{mod} \,\,C_\p) .\]
It follows that $L_\p = (C + g_m^{\mathrm{odd}})_\p$.  Since 
\[f_mg^{\rm{even}}_m=   y_mg_m^{\rm{odd}} - (y_m^2-zf_m^2) g^{\rm{odd}}_{m-1} \in C + (g_m^{\mathrm{odd}}),\]
we have
\[L_\p = (C + (g_m^{\mathrm{odd}}))_\p = J_\p.\]

Since $Q_\p$ is a complete intersection, in particular $T_\p/Q_\p$ is Cohen-Macaulay.  Since $J_\p = C_\p:Q_\p$, we have $T_\p/J_\p$ is also Cohen-Macaulay.  In particular, $J_\p$ is unmixed of height $m$; therefore $\p T_\p \notin \ass(T_\p/J_\p)$ and so $\p \notin \ass(T/J)$.
\end{proof}

We can now prove the following:

\begin{Proposition}\label{C:Q} 
With notation as above, $L = J$, i.e. $C:Q = C + (g_m^{\mathrm{odd}}, g_m^{\mathrm{even}})$.
\end{Proposition}

\begin{proof}
The containment $L\supseteq J$ follows from Lemma~\ref{primdec} and Corollary~\ref{evenoddcor}. Next we show $J^{un}=L$. 
Since $C=Q\cap L \subseteq J \subseteq L$ all have height $m$, and $Q,L$ are unmixed, then $C \subseteq J^{un}\subseteq L$. Since $C\subseteq J^{un}$ are unmixed of the same height, then $\Ass(T/J^{un})\subseteq \Ass(T/C)$, so, by Lemma \ref{primdec}(4), all associated primes of $T/J^{un}$ have the form $\RLP(f_1,\pm f_2,\ldots,\pm f_m)$. By Theorem \ref{RLbasics} (or the proof of Lemma \ref{nop}) they are all contained in $\p=(y_1,\ldots,y_m,z)$. Since $J_{\p}=L_{\p}$, by Lemma \ref{nop}, then $J_{Q_i}=L_{Q_i}$ for each $Q_i \in \Ass(T/J^{un})$. This proves $J^{un} = L$.

It then suffices to prove that $J$ is unmixed. We observe that for any associated prime $q$ of $T/J$ we have $\h(q)\leq m+1$, because $$\h(q)\leq \pd(T/J)\leq \pd(T/\In_<(C+(g_m^{\rm{even}}, g_m^{\rm{odd}}))) \leq m+1.$$ The first inequality follows from \cite[Lemma~2.6]{HMMS}, the second inequality follows from \cite[Theorem 22.9]{Peeva}, and the last inequality is proved in Lemma~\ref{init1}.  Therefore, we only need to prove that $J$ contains no associated primes of height $m+1$. Our next goal is to prove the following\\

\noindent{\textbf{ Claim 1.}} There exists a linear form $ \ell$ in $k[y_1,\ldots,y_m]$ that is regular on $T/J$.\\

{\textit{ Proof of Claim 1.}}  It suffices to show that no prime ideal $\p\in \Ass(T/J)$ of height $m+1$ contains $(y_1,\ldots,y_m)$. Indeed, if any such $\p$ exists, since $C\subseteq \p$, then one has $z(f_1^2,\ldots,f_m^2)\subseteq \p$; since $\h(f_1^2,\ldots,f_m^2)=\h(I)>1$, the only possibility is that $z\in \p$, and therefore $\p=(y_1,\ldots,y_m,z)$.  But this possibility is ruled out by Lemma~\ref{nop}. $\blacksquare$ \\

\noindent {\textbf{Claim 2.}} We may assume $y_m$ is regular on $T/J$. \\

{\textit{ Proof of Claim 2.}} By Claim 1 there is a linear form $0\neq \ell \in k[y_1,\ldots,y_m]$ that is regular on $T/J$.  By possibly multiplying by a unit and permuting the variables, we may assume that $\ell = y_m + \sum_{i=1}^{m-1} \alpha_i y_i$, where $\alpha_i \in k$.  We consider the automorphism $\psi$ of $T$ that fixes all variables except it sends $y_m \mapsto \ell$.  It is easy check that $\psi^{-1}(J)$ has the same generators as $J$ except that every instance of $f_m$ is replaced by $f_m + \sum_{i = 1}^{m-1} \alpha_i f_i$.  This then corresponds to choosing a different minimal set of generators of $I$ before constructing the Rees-like prime.  Since $\ell$ is not in any associated prime of $J$, $y_m$ is not in any associated prime of $\psi^{-1}(J)$. $\blacksquare$ \\
\\
We now conclude the proof of Proposition~\ref{C:Q}. Since $y_m$ is regular on $T/J$ and $y_m^2-zf_m^2\in J$, then also $f_m$ is regular on $T/J$. To prove $J$ is unmixed it then suffices to show $J_{f_m}$ is unmixed. 
Since $f_m$ is a unit in $T_{f_m}$ and $f_mg^{\rm{even}}_m= y_mg_m^{\rm{odd}} -(y_m^2-zf_m^2) g_{m-1}^{\rm{odd}} \in (C+(g_m^{\rm{odd}}))_{f_m}$, the ideal $J_{f_m}=(C+(g_m^{\rm{odd}}))_{f_m}$ is an almost complete intersection of height $m$.

Now, in $T_{f_m}$ we have
\be
\begin{array}{ll}
\phantom{push stuff} J+(y_m) 
 & =(y_1^2-zf_1^2,\ldots,y_{m-1}^2-zf_{m-1}^2, y_m^2-zf_m^2,y_m,g_m^{\rm{odd}})\\
		& = (y_1^2-zf_1^2,\ldots,y_{m-1}^2-zf_{m-1}^2, zf_m^2,y_m,g_m^{\rm{odd}})\\
(\mbox{because }f_m \mbox{ is a unit})	& = (y_1^2-zf_1^2,\ldots,y_{m-1}^2-zf_{m-1}^2, z,y_m,g_m^{\rm{odd}})\\
		& = (y_1^2,y_2^2,\ldots,y_{m-1}^2,y_m,z,g_m^{\rm{odd}})\\
(\mbox{by definition of }g_m^{\rm odd} )	& = (y_1^2,y_2^2,\ldots,y_{m-1}^2,y_m,z,y_1\cdots y_{m-1}).
\end{array}
\ee
Since $M = (y_1^2,\ldots,y_{m-1}^2,y_1y_2\cdots y_{m-1})$ is $(y_1,\ldots,y_{m-1})$-primary and extended from $k[y_1,\ldots,y_m]$, $M$ is Cohen-Macaulay of height $m-1$. Since $y_m,z$ is a regular sequence on $(T/M)_{f_m}$, the ideal $(y_1^2,y_2^2,\ldots,y_{m-1}^2,y_1\cdots y_{m-1},y_m,z)_{f_m} = (J+(y_m))_{f_m}$ is Cohen-Macaulay too. Since $y_m$ is regular on $T/J$ and $f_m$ is regular on $T/J$, $y_m$ is also regular on $(T/J)_{f_m}$, and thus $(T/J)_{f_m}$ is Cohen-Macaulay.  In particular, $J_{f_m}$ is unmixed and then so is $J$.

\end{proof}

We are now able to construct a finite $T$-free resolution of the canonical module of any Rees-like algebra $\R(I) = S[It, t^2] = T/\RLP(I)$, assuming $\mathrm{char}(k) \neq 2$.  It is built from an amalgamation of the Koszul complexes on the generators $f_1,\ldots,f_m$ of $I$ and the variables $y_1,\ldots,y_m$.

\
\begin{Theorem}\label{canonicalModule} Suppose $k$ is a field with $\mathrm{char}(k) \neq 2$.  Let $S = k[x_1,\ldots,x_n]$ and let $I = (f_1,\ldots,f_m)$ be an ideal of $S$ with $\h(I) \ge 2$.  Then the canonical module $\omega_{\R(I)}$ of the Rees-like algebra $\R(I)$ is Cohen-Macaulay.
In particular, if $\matM$ is the matrix
\[ \matM =  \left[\begin{tabular}{cccc|cccc}
$y_1$ & $y_2$ & $\cdots$ & $y_m$ & $f_1$ & $f_2$ & $\cdots$ & $f_m$ \\
\hline
$z f_1$ & $z f_2$ & $\cdots$ & $z f_m$ & $y_1$ & $y_2$ & $\cdots$ & $y_m$ 
\end{tabular}\right],\]
then the canonical module of the Rees-like algebra $\R(I)$ is
\[\omega_{\R(I)} \cong \mathrm{coker}(\matM),\]
as $T$-modules, and thus $\mathrm{type}(\R(I)) = 2$.
\end{Theorem}

\begin{proof} As usual let $T = S[y_1,\ldots,y_m,z]$.
Let $\mathbf{K}_\bullet(\underline{y})$ denote the Koszul complex on $y_1,\ldots,y_m$ over $T$ with differential maps $d_i^{\underline{y}}:\mathbf{K}_i(\underline{y}) \to \mathbf{K}_{i-1}(\underline{y})$, and let $\mathbf{K}_\bullet(\underline{f})$ denote the Koszul complex on $f_1,\ldots,f_m$ over $T$ with differential maps $d_i^{\underline{f}}:\mathbf{K}_i(\underline{f}) \to \mathbf{K}_{i-1}(\underline{f})$.  Define a new complex of free $T$-modules $\mathbf{D}_\bullet$ with $\mathbf{D}_i = T^{2 \binom{m}{i}}$ for $0 \le i \le m$ with differential given as a matrix by
\[\arraycolsep=1.6pt\def\arraystretch{1.6}
d_i^\mathbf{D} =  \left[\begin{tabular}{c|c} $d_i^{\underline{y}}$ & $d_i^{\underline{f}}$\\
\hline $z \!\cdot\!  d_i^{\underline{f}}$ & $d_i^{\underline{y}} $
\end{tabular}\right] .\]
It is easy to check that $d_{i-1}^\mathbf{D} \circ d_i^\mathbf{D} = 0$ and thus $\mathbf{D}_\bullet$ is a complex.  We also have the following short exact sequences of complexes
\[ 0 \to \mathbf{D}_\bullet \xrightarrow{z} \mathbf{D}_\bullet \to \mathbf{D}_\bullet/z\mathbf{D}_\bullet \to 0.\]
and
\[ 0 \to \mathbf{K}_\bullet(\underline{y}) \to  \mathbf{D}_\bullet/z\mathbf{D}_\bullet \to \mathbf{K}_\bullet(\underline{y}) \to 0.\]
Because $\mathbf{K}_\bullet(\underline{y})$ is acyclic, it follows from the long exact sequence of homology associated to the second short exact sequence that $ \mathbf{D}_\bullet/z\mathbf{D}_\bullet$ is also acyclic.  Now from the long exact sequence associated to the first short exact sequence we see that multiplication by $z$ induces an isomorphism on $H_i( \mathbf{D}_\bullet)$ for $i > 0$;  then by Nakayama's Lemma we get $H_i( \mathbf{D}_\bullet ) = 0$ for $i > 0$.  Note that $d_1^\mathbf{D} = \matM$.  

Now define $d_0^\mathbf{D}: \mathbf{D}_0 \to \frac{C:Q}{C}$ as follows.  By Proposition~\ref{C:Q}, $\frac{C:Q}{C}$ is minimally generated by $g_m^{\rm{even}}$ and $-g_m^{\rm{odd}}$.  Since $\mathbf{D}_0 = T^2$, we map the first basis element to $g_m^{\rm{even}}$ and the second basis element to $-g_m^{\rm{odd}}$.  By Lemma~\ref{identities}, we have
\[y_m g^{\rm{even}}_m + z f_m (-g^{\rm{odd}}_m) = g^{\rm{even}}_{m-1} \left(y_m^2 - z f_m^2\right)  \in C \]
and
\[f_m g^{\rm{even}}_m + y_m (-g^{\rm{odd}}_m) = - g^{\rm{odd}}_{m-1} \left(y_m^2 - z f_m^2\right)  \in C.\]
Therefore $\im(d_1^\mathbf{D}) = \im(\matM) \subseteq \ker(d_0^\mathbf{D})$.  To show the reverse inclusion, suppose that $a,b \in T$ such that $d_0^\mathbf{D}[a,b]^\mathsf{T} = 0 \in \frac{C:Q}{C}$; that is, 
\[a \cdot g^{\rm{even}}_m + b (-g^{\rm{odd}}_m) \in C.\]
Then by Lemma~\ref{SES}, $a \in (C + (g^{\rm{odd}}_m)):(g^{\rm{even}}_m) = IT + (y_1,\ldots,y_m)$.  Since the entries in the first row of $\matM$ generate $IT + (y_1,\ldots,y_m)$, we can use the columns of $\matM$ to rewrite $a$ and $b$ and we may assume that $a = 0$.  But then $b \in C:(g^{\rm{odd}}_m) = Q$.  By Theorem~\ref{RLbasics}, every element of $Q$ is a linear combination of the elements $y_iy_j - zf_if_j$, where $1 \le i \le j \le m$ and $\sum_j c_j y_j$, where $\sum_j c_j f_j = 0$.  Note that
\[
\begin{bmatrix} 0 \\ y_i y_j - z f_i f_j  \end{bmatrix} = y_j \begin{bmatrix} f_i  \\ y_i \end{bmatrix} - f_i \begin{bmatrix} y_j \\ z f_j \end{bmatrix} \in \im (d_1^\mathbf{D}),\]
and
\[
\begin{bmatrix} 0 \\ \sum_j c_j y_j  \end{bmatrix} = \sum_j c_j \begin{bmatrix} f_j  \\ y_j \end{bmatrix}  \in \im (d_1^\mathbf{D}),\]
where $\sum_j c_j f_j = 0$.  Therefore $[0,b]^\mathsf{T} \in \im (d_1^\mathbf{D})$, for any $b \in Q$.  It follows that $\im(d_1^\mathbf{D}) =  \ker(d_0^\mathbf{D})$, and that $\mathbf{D}_\bullet$ is a minimal $T$-free resolution of $\frac{C:Q}{C}$.  Finally, we have $\omega_{\R(I)} \cong \frac{C:Q}{C}$, e.g. by \cite[Lemma 3.1]{HMMS}.

\end{proof}

In retrospect, perhaps the fact that the canonical module is Cohen-Macaulay should not be surprising since the integral closure of $S[It, t^2]$ is a polynomial ring, and thus a finite Cohen-Macaulay module over the non-Cohen-Macaulay Rees-like algebra.  Yet, we find the self-dual nature of the $T$-free resolution of the canonical module in the previous theorem surprising.  The authors plan to study more generally non-Cohen-Macaulay rings whose canonical modules are self dual in a future paper. 

As a corollary, we get the following surprising self-duality statement:

\begin{Corollary} Using the notation above, 
\[\omega_{\R(I)} \cong \ext^m_{_T}(T/Q,T) \cong\ext^m_{_T}(\omega_{\R(I)},T).\]

\end{Corollary}

\begin{proof} Because $\mathbf{K}_\bullet(\underline{y})$ and $\mathbf{K}_\bullet(\underline{f})$ are self-dual, it follows from the definition that $\mathbf{D}_\bullet$ is self-dual as well, i.e. $\mathbf{D}_\bullet \cong \mathrm{Hom}_{_T}(\mathbf{D}_\bullet,T)$.

\end{proof}

\begin{Example}
Let $S = k[x_1,x_2]$ and set $I = (x_1,x_2)^2$. We construct the resolution of the canonical module of the Rees-like algebra $\R(I)$. As such, set $T = S[y_1,y_2,y_3,z]$ and let $Q = \RLP(x_1^2,x_1x_2,x_2^2)$.  By the previous theorem, $\omega_{\R(I)} \cong \frac{C:Q}{C}$, where $C = (y_1^2 - zx_1^4, y_2^2-zx_1^2x_2^2, y_3^2-zx_2^4)$ and $$C:Q = C + (g_3^{\mathrm{odd}}, g_3^{\mathrm{even}}),$$ where
\[g_3^{\mathrm{even}} = y_1y_2y_3 + x_1x_2^3y_1z + x_1^3x_2y_2z + x_1^3x_2y_3z,\]
\[g_3^{\mathrm{odd}} = x_2^2y_1y_2 + x_1x_2y_1y_3 +x_1^2y_2y_3 + x_1^3x_2^3z.\]
Moreover, as a $T$-module, $\omega_{\R(I)}$ has $T$-free resolution:
\[
\xymatrix{
T^2 & \ar[l]_{d_1} T^6 & \ar[l]_{d_2} T^6 & \ar[l]_{d_3} T^2 & \ar[l] 0,
}\]
where
\begin{align*}
d_1 &=  \left[\begin{tabular}{ccc|ccc}
${y}_{1}$&
      ${y}_{2}$&
      ${y}_{3}$&
      $ {x}_{1}^2$&
      $ x_1{x}_{2}$&
      $ {x}_{2}^2$\\
      \hline
      $z{{x}_{1}^2}$&
      $zx_1{{x}_{2}}$&
      $z{{x}_{2}^2}$&
      ${{y}_{1}}$&
      ${{y}_{2}}$&
      ${{y}_{3}}$\\
      \end{tabular}\right]\\
      d_2 &= \left[\begin{tabular}{ccc|ccc}
      ${-{y}_{2}}$&
      ${-{y}_{3}}$&
      $0$&
      ${- x_1{x}_{2}}$&
      ${- {x}_{2}^2}$&
      $0$\\
      $\phantom{-}{y}_{1}$&
      $0$&
      ${-{y}_{3}}$&
      $\phantom{-} {x}_{1}^2$&
      $0$&
      ${- {x}_{2}^2}$\\
      $0$&
      $\phantom{-}{y}_{1}$&
      $\phantom{-}{y}_{2}$&
      $0$&
      $\phantom{-} {x}_{1}^2$&
      $\phantom{-} x_1{x}_{2}$\\
      \hline
      ${-zx_1{x}_{2}}$&
      ${-z{x}_{2}^2}$&
      $0$&
      ${-{y}_{2}}$&
      ${-{y}_{3}}$&
      $0$\\
      $\phantom{-}z{x}_{1}^2$&
      $0$&
      ${-z{x}_{2}^2}$&
      $\phantom{-}{y}_{1}$&
      $0$&
      ${-{y}_{3}}$\\
      $0$&
      $\phantom{-}z{x}_{1}^2$&
      $\phantom{-}zx_1{x}_{2}$&
      $0$&
      $\phantom{-}{y}_{1}$&
      $\phantom{-}{y}_{2}$\\
     \end{tabular}\right]\\
      d_3 &=  \left[\begin{tabular}{c|c}
      ${-{y}_{3}}$&
      $- {x}_{2}^2$\\
      $\phantom{-}{y}_{2}$&
     $\phantom{-} { x_1{x}_{2}}$\\
      ${-{y}_{1}}$&
      $- {x}_{1}^2$\\
      \hline
      ${-z{x}_{2}^2}$&
      $-{y}_{3}$\\
     $\phantom{-} zx_1{x}_{2}$&
     $\phantom{-} {{y}_{2}}$\\
      ${-z{x}_{1}^2}$&
      $-{y}_{1}$\\
\end{tabular}\right].
\end{align*}
\end{Example}

\section*{Acknowledgements}

 The second author was supported by a grant from the Simons Foundation (576107, JGM).


\begin{thebibliography}{AAAA}

\bibitem{AH} T. Ananyan and M. Hochster, \emph{Small subalgebras of polynomial rings and Stillman's conjecture}, arXiv:1610.09268v1. 

\bibitem{BE} D. Buchsbaum and D. Eisenbud, \emph{What makes a complex exact?}, J. Algebra, {\bf 25}, (1973), 259--268.

\bibitem{BE74} D. Buchsbaum and D. Eisenbud, \emph{Some structure theorems for finite free resolutions}, Adv. Math., {\bf 12}, (1974), 84--139.

\bibitem{E} D. Eisenbud, \emph{Commutative Algebra with a view toward Algebraic Geometry}, Graduate Texts in Math., {\bf  150}, Springer-Verlag, Berlin and New York, 1995.

 \bibitem{EG} D. Eisenbud and S. Goto, \emph{Linear free resolutions and minimal multiplicity},  J. Algebra, {\bf  88}, (1984),  89--133.

\bibitem{EHU}D. Eisenbud, C. Huneke, and B. Ulrich, \emph{The regularity of Tor and graded Betti numbers}, Amer. J. Math., {\bf 128}, (2006), 573--605.


 \bibitem{Fl}  H. Flenner, \emph{Die Satze von Bertini fur lokale Ringe}, Math. Ann., {\bf 229},  (1977), 97--111. 


\bibitem{GLP}   L. Gruson, R. Lazarsfeld, and C. Peskine, \emph{On a theorem of Castelnuovo and the equations defining projective varieties}, Invent. Math., {\bf 72}, (1983), 491--506.


\bibitem{HerzogEne} V. Ene and J. Herzog, \emph{Gr\"obner bases in commutative algebra}, Graduate Studies in Mathematics, {\bf 130}. American Mathematical Society, Providence, RI, 2012.

\bibitem{HMNU} C. Huneke, J. C. Migliore, U. Nagel, and B. Ulrich, \emph{Minimal homogenous liaison and licci ideals}, Cont. Math., {\bf 448}, (2007), 129--139.


\bibitem{HMMS} C. Huneke, P. Mantero, J. McCullough and A. Seceleanu, \emph{Multiple structures with arbitrarily large projective dimension supported on linear subspaces}, J. Algebra, {\bf 447}, (2016), 183--205.

  \bibitem{KP}  S. Kwak and  J. Park, \emph{A bound for Castelnuovo-Mumford regularity by double point divisors},
arXiv: 1406.7404v1. 


\bibitem{Laz} R. Lazarsfeld, \emph{A sharp Castelnuovo bound for smooth surfaces}, Duke Math. J. {\bf 55}, (1987), 423--438.

\bibitem{LeVi} J. V. Leahy and M. A. Vitulli, \emph{ Seminormal rings and weakly normal varieties}, Nagoya Math. J.,
{\bf 82}, (1981), 27--56.

\bibitem{M} J. McCullough, \emph{On the Maximal Graded Shifts of Ideals and Modules}, to appear in J. Algebra.

\bibitem{MP} J. McCullough and I. Peeva, \emph{Counterexamples to the Eisenbud--Goto regularity conjecture},  J. Amer. Math. Soc. {\bf 31} (2018), no. 2, 473--496.

\bibitem{Niu} W. Niu, \emph{Castelnuovo–Mumford regularity bounds for singular surfaces}, Math.e Zeit., {\bf 280}, 3-4, 609--620.

\bibitem{M2} D. Grayson and M. Stillman, \emph{Macaulay2, a software system for research in algebraic geometry},
        Available at http://www.math.uiuc.edu/Macaulay2/.

\bibitem{Noma} A. Noma, \emph{Generic inner projections of projective varieties and an application to the positivity of double point divisors},  Trans. Amer. Math. Soc., {\bf 366}, (2014),  4603--4623.

\bibitem{Peeva} I. Peeva, \emph{Graded Syzygies}, Algebra and Applications, {\bf 14}. Springer-Verlag London, Ltd., London, 2011.
   
\bibitem{P} H. Pinkham, \emph{A Castelnuovo bound for smooth surfaces}, Invent. Math., {\bf 83}, (1986), 321--332.


\bibitem{Ran} Z. Ran, \emph{Local differential geometry and generic projections of threefolds}, J. Differential Geom., {\bf 32}, (1990), 13--137.

\bibitem{SS} B. Serbinowski and K. Schwede, \emph{Seminormalization: a package for computing seminormalization of rings}, Version~0.1,  
        Available at https://github.com/Macaulay2/M2/tree/master/M2/Macaulay2/packages.
      

\bibitem{Vasconcelos} W. Vasconcelos, \emph{On linear complete intersections}, J. Algebra, {\bf 111}, (1987), no. 2, 306--315. 

\bibitem{Vasconcelos2} W. Vasconcelos, \emph{Computational methods in commutative algebra and algebraic geometry},
 Algorithms and Computation in Mathematics, {\bf 2}. Springer-Verlag, Berlin, 1998.

\bibitem{Vi} M. A. Vitulli, \emph{Weak normality and seminormality}, Commutative algebra -- Noetherian and non-Noetherian perspectives, 441-- 480, Springer, New York, 2011.

\bibitem{Ya} H. Yanagihara, \emph{Some results on weakly normal ring extensions},  J. Math. Soc. Japan, {\bf 35}, (4),
(1983), 649--661.




\end{thebibliography}
\end{document}